\documentclass[11pt, a4paper, reqno]{amsart}
\usepackage{booktabs}
\usepackage[subnum]{cases}
\usepackage{algorithm}
\usepackage{algorithmic}
\usepackage{amsfonts}
\usepackage{amsmath}
\usepackage{amssymb}
\usepackage{amsthm}
\usepackage[english]{babel}
\usepackage{bbold}
\usepackage{color}
\usepackage{comment}
\usepackage{csquotes}
\usepackage{enumerate}
\usepackage[margin=1.5in]{geometry}
\usepackage{graphicx}
\usepackage[hidelinks]{hyperref}
\usepackage{mathdots}
\usepackage{mathtools}
\usepackage{tikz-cd}
\usepackage{verbatim}
\usepackage[all]{xy}
\usepackage{xcolor}
\usepackage{zref-clever}
\usepackage{subcaption}

\usepackage{mysansmath}
\let\TreeFnt\mathsfbfit
\zcsetup{cap}

\usepackage[style=alphabetic]{biblatex}
\AtEveryBibitem{%
	\clearfield{eventtitle}%
	\clearfield{venue}%
	\clearfield{eventyear}%
	\iffieldundef{doi}{}{%
		\clearfield{issn}%
		\clearfield{isbn}%
		\clearfield{url}%
	}%
}

\usetikzlibrary{calc}

\usepackage[patch-newtheorem]{mytheorem-enum}
\newtheorem{theorem}{Theorem}[section]
\newtheorem{corollary}[theorem]{Corollary}
\newtheorem{lemma}[theorem]{Lemma}
\newtheorem{proposition}[theorem]{Proposition}
\newtheorem{question}[theorem]{Question}

\theoremstyle{definition}
\newtheorem{definition}[theorem]{Definition}
\theoremstyle{remark}
\newtheorem{example}[theorem]{Example}
\newtheorem{remark}[theorem]{Remark}

\newtheoremstyle{claim}{2pt}{2pt}{}{}{\itshape}{.}{.5em}{}
\theoremstyle{claim}
\newtheorem{claim}{Claim}
\zcRefTypeSetup{question}{
	Name-sg = Question,
	name-sg = question,
	Name-pl = Questions,
	name-pl = questions,
}
\zcRefTypeSetup{claim}{
	Name-sg = Claim ,
	name-sg = claim,
	Name-pl = Claims ,
	name-pl = claims ,
}
\zcRefTypeSetup{equation}{
	Name-sg=,
	name-sg=,
	Name-pl=,
	name-pl=,
	Name-sg-ab=,
	name-sg-ab=,
	Name-pl-ab=,
	name-pl-ab=
}
\AtBeginEnvironment{proof}{\setcounter{claim}{0}}
\makeatletter
\newenvironment{claimproof}[1][Proof of claim]{\par
	\pushQED{}
	\normalfont \topsep2\p@\@plus2\p@\relax
	\trivlist
	\item\relax
	{\itshape
		#1\@addpunct{.}}\hspace\labelsep\ignorespaces
}{%
	\popQED\endtrivlist\@endpefalse
}
\makeatother
\newcommand{\RR}{\mathbb{R}}
\newcommand{\TT}{\mathbb{T}}

\newcommand{\bergman}[1]{B(K_{#1})}
\newcommand{\one}{\mathbb{1}}
\newcommand{\conv}{{\operatorname*{conv}}}

\newcommand{\torus}[1]{\mathbb{R}^{#1} / \mathbb{R}\one}
\newcommand{\tropdist}{d_{\Delta}}

\newcommand{\symdist}{d_{\sym}}
\newcommand{\tropmedian}[1]{\mu_{#1}}

\newcommand{\Oscar}{\textsc{oscar}}
\newcommand{\Polymake}{\textsc{polymake}}

\newcommand{\MCF}{\textsc{MCF}}
\makeatletter
\newcommand\@mybf{bx}
\newcommand\K{%
	\ifx\f@series\@mybf\relax\boldmath\fi
	\texorpdfstring{$k$}{k}%
}

\usepackage[colorinlistoftodos,bordercolor=orange,backgroundcolor=orange!20,linecolor=orange,textsize=tiny]{todonotes}

\title{Tropical \K-means clustering for phylogenetic trees}

\author[F. Lenzen]{Fabian Lenzen}
\address[Fabian Lenzen]{Technische Universität Berlin}
\email{lenzen@math.tu-berlin.de}
\author[L. Weis]{Lena Weis}
\address[Lena Weis]{Technische Universität Berlin}
\email{weis@math.tu-berlin.de}
\thanks{The authors want to thank Michael Joswig for proposing the topic.
  The topic was inspired by a question of Rodrigo Silveira at the XXI Spanish Meeting on Computational Geometry regarding the possibility of tropical $k$-means.
  We also want to thank Shelby Cox for pointing out \cite{Smith:2021} to us.
}
\thanks{
  FL was funded by the Deutsche Forschungsgemeinschaft (DFG, German Research
  Foundation) under Germany's Excellence Strategy – The Berlin Mathematics
  Research Center MATH+ (EXC-2046/1, EXC-2046/2, project ID: 390685689).
}

\DeclareMathOperator{\sym}{sym}

\DeclareMathOperator{\Vor}{Vor}
\DeclareMathOperator{\argmin}{arg\,min}
\DeclareMathOperator{\FW}{FW}
\DeclareMathOperator{\Mid}{mid}
\DeclarePairedDelimiter{\abs}{\lvert}{\rvert}
\DeclarePairedDelimiter{\norm}{\lVert}{\rVert}

\DeclarePairedDelimiterX{\Set}[1]{\{}{\}}{\setargs{#1}}
\NewDocumentCommand{\setargs}{>{\SplitArgument{1}{;}}m}{\setargsaux#1}
\NewDocumentCommand{\setargsaux}{mm}{\IfNoValueTF{#2}{#1} {#1\nonscript\:\delimsize\vert\allowbreak\nonscript\:\mathopen{}#2}}%
\usetikzlibrary{scopes,calc,quotes}
\let\Loss\ell

\usepackage{enumitem}
\setlist{ref=(\arabic*)}

\usepackage{marginnote,xpatch}
\makeatletter
\patchcmd{\@todonotes@drawMarginNoteWithLine}{\marginpar}{\marginnote}{}{}
\makeatother
\begin{document}
  \begin{abstract}
    The asymmetric tropical distance is a distance measure on the tropical torus $\torus{n}$
    and in particular on the Bergman fan $\bergman{N} \subseteq \torus{\binom{N}{2}}$ of the complete graphical matroid.
    In this paper, we define and analyse a clustering algorithm for equidistant phylogenetic trees based on this distance,
    using the correspondence between $\bergman{N}$ and the space of equidistant trees with $N$ leaves.
  \end{abstract}
  
  \maketitle
  
    
  \section{Introduction}
  In mathematical biology, an active research problem is to find, given a set $\mathcal{S}$ of phylogenetic trees, a so-called \emph{consensus tree}
  that, intuitively, represents the trees in $\mathcal{S}$ by a single one.
  There exist various candidates for consensus trees and different \emph{consensus methods} to compute them; see \cite{Bryant:2003} for a survey.
  
  However, computing a consensus tree only makes sense if the trees in $\mathcal{S}$
  are \enquote{similar enough} to be unified.
  Otherwise, if the set of trees is too heterogeneous, it becomes necessary to partition $\mathcal{S}$ into subsets of \enquote{sufficiently similar} trees,
  and represent each of these by an individual consensus tree.
  To accomplish this, we apply a method called $k$-means clustering for this task.
  We briefly summarize the method.
  
  \smallskip
  Given a finite subset $S$ of \emph{sites} in a metric space $(V, d)$, an integer $k \leq \abs{S}$ and a positive real number $q$,
  the \emph{$k$-means clustering} method seeks to partition the set $S$ into a set $\mathcal{C} = \{C_1,\dotsc,C_k\}$ of subsets of $S$, called \emph{clusters},
  such that the \emph{loss}
  \begin{equation}
  	\label{eq:k-means-clustering}
  	\Loss(\mathcal{C}) \coloneqq \sum_{i=1}^k \sum_{v \in C_i} d(v, \mu_i)^q
  \end{equation}
  becomes minimal \cite{Steinhaus:1957}.
  Here, $\mu_i$ denotes the \emph{$q$-mean} of $C_i$ in $V$; i.e., the (or a) point $\mu_i \in V$ for which
  \begin{equation}
  	\label{eq:euclidean-median}
  	\sum_{v \in C_i} d(v, \mu_i)^q
  \end{equation}
  is minimized.
  In the context of clustering, the representative $\mu_i$ of the cluster $C_i$ is also called the \emph{centroid} of $C_i$.
  
  For example, if $V$ is an euclidean space with euclidean metric $d(s,t) = \norm{s-t}$ for the $L^2$-norm $\norm{-}$,
  $\mu_i$ is the \emph{center of mass} of $C_i$, given as the the coordinate-wise $2$-mean of the vectors in $C_i$.
  
  In general, finding the local optimum $\mathcal{C}^*$ of \zcref{eq:k-means-clustering} in arbitrary metric spaces is NP-complete \cite[Theorem~8]{KettleboroughRayward-Smith:2013}.
  While in principle, \zcref{eq:k-means-clustering} with $q=2$ can be solved in $d$-dimensional Euclidean spaces in time $\mathcal{O}(N^{(kd+1)})$ \cite{InabaKatohEtAl:1994},
  the large exponent usually renders this algorithm unfeasible in practice.
  
  However, a local optimum $\mathcal{C}$ of \zcref{eq:k-means-clustering} (for any $q$) can be obtained
  through a simple iterative process called \emph{Lloyd's algorithm} \cite[\S3.6]{MacQueen:1967}.
  By seeding the initial state of this iterative process appropriately,
  it is possible to bound the expectation value of the \emph{competitive factor} $\Loss(\mathcal{C})/\Loss(\mathcal{C}^*)$.
  With this seeding, the algorithm is usually called \emph{$k$-means++ clustering} \cite{ArthurVassilvitskii:2007}; see \zcref{sec:tropical-k-means} for details.
  
  For $q = 1$ in \zcref{eq:euclidean-median}, the $1$-mean $\mu_i$ is also called a \emph{Fermat--Weber point} of $C_i$.
  In euclidean spaces, if $C_i$ is in general position, the Fermat--Weber point of $C_i$ is unique and called the \emph{median} of $C_i$.
  While no explicit formula for the euclidean median exists, it can be approximated to arbitrary accuracy through an iterative process;
  e.g., through \emph{Weiszfeld's algorithm} \cite{Weiszfeld:1937}.
  
  Observe however that neither the definition of the $q$-mean,
  nor the definition of the $k$-means clustering, depend on $V$ being a metric space.
  In fact, it is even perfectly natural to generalize the problem to quasi-metric spaces $V$.
  
  \smallskip
  This aspect of generalizability becomes vital when it comes to applying the method to phylogenetic trees.
  The set of trees with a fixed set of $N$ taxa has an inherent structure as a \emph{tropical subspace} of $\RR^n$ for $n = \binom{N}{2}$.

  The latter is equipped with the \emph{asymmetric tropical distance}
  \begin{equation}
    \label{eq:def-tropdist}
    \begin{aligned}
      \tropdist(x,y) \ = \ \sum_{i=1}^{n} (y_i - x_i) + n \max\{x_i - y_i\} \ ,
    \end{aligned}
  \end{equation}
  which defines a \emph{quasi-metric} on $\RR^n$;
  i.e., it is positive definite (meaning that $\tropdist(x,y)=0$ if and only if $x=y$) 
  and satisfies the triangle inequality
  \begin{equation}
    \label{eq:tropdist-triangle-ineq}
    \tropdist(x, z) \leq \tropdist(x, y) + \tropdist(y, z),
  \end{equation}
  but is not symmetric.  

  A method for computing a mean with respect to $\tropdist$ (for $q=1$) is given in \cite{ComaneciJoswig:24},
  which gives rise to the \emph{tropical median consensus method}. 
  With this, we develop and study a method for \emph{tropical $k$-means clustering} with respect to the asymmetric tropical distance.
  In the following, we briefly summarize our main results.
  \smallskip
  
  \subsection{Our contribution}
  We propose an algorithm for computing tropical $k$-means clustering based on Lloyd's iterative algorithm.
  Analogously to the euclidean setting, tropical $k$-means clustering can converge to local optimum $\mathcal{C}$
  rather than the global optimum $\mathcal{C}^*$ (\zcref{ex:local-optimum}).
  Using \cite{NielsenSun:2019} we show that with the $k$-means++ initial seeding mentioned above, the expectation value of the competitive factor becomes bounded (\zcref{thm:competitive}).
  Finally, we show that if $S$ consists of binary trees of sufficiently different combinatorial type,
  then tropical $k$-means++ clustering will separate them accordingly (\zcref{thm:tropicalk-means}).
  
  We provide an implementation of our algorithm in the \Oscar\ computer algebra system.
  We demonstrate the efficacy of our method by clustering different sets of phylogenetic trees (synthetic and non-synthetic).
  
  \subsection{Structure of the paper}
  The paper is stuctured as follows.
  We begin by introducing the necessary notions in \zcref{sec:preliminaries}. 
  We continue with our main results, including the tropical $k$-means algorithm, in \zcref{sec:tropical-k-means}. 
  The details of our implementation of tropical $k$-means clustering in \Oscar\ are stated in \zcref{sec:implementation}.
  Last, an outlook is provided in \zcref{sec:outlook}.
    

  \subsection{Related work}
  Clustering phylogenetic trees via $k$-means has been first considered by \textcite{StockhamEtAl2002} with respect to the Robinson--Foulds distance.
  This metric measures the distance of two phylogenetic trees as the number of different splits between them.
  The distance of two trees is therefore an integer value.
  The algorithm of Stockham et al. was further developed by \textcite{TahiriEtAl:2022} to allow for a dataset of trees with different sets of taxa.
  A qualitative anaylsis of different clustering methods of phylogenetic trees with respect to multiple tree distances intended to quantify the similarities of subtrees has been conducted by \textcite{Smith:2021}.
  In his paper, Smith points out that ``the Robinson–Foulds tree spaces [...] often fail to group trees according to phylogenetic similarity'' as it lacks resolution
  which demonstrates the need for a different approach.
  The tropical asymmetric distance on the other hand shows high resolution which is why we propose this alternative clustering.
  
  The tropical asymmetric distance $\tropdist$ was first introduced by \textcite{AminiManjunath2010}.
  \textcite{ComaneciJoswig:24} studied Fermat--Weber sets with respect to $\tropdist$ and introduced the consensus method which our paper relies on.
  Finally, we want to mention two key contributions connecting tropical geometry and phylogenetic trees, namely \cite{DS04} and \cite{ArdilaKlivans}.

  \section{Phylogenetic trees from a tropical perspective}\label{sec:preliminaries}
  \renewcommand\t{\TreeFnt{t}}
  In this section we introduce the main tools required for tropical $k$-means clustering. 
  We begin by introducing our main object of study: \emph{phylogenetic trees}.
  For us, a phylogenetic tree is a rooted tree $\t = (V,E)$ with nodes $V$ and edges $E$ whose leaves are labeled by a set of \emph{taxa} $X$,
  whose edges are assigned a length.
  In particular, internal edges are assigned a positive length, while external edges may be assigned negative length.
  The \emph{depth} of a node $v$ is the length of the path from the root to $v$.
  A tree is called \emph{equidistant} if all leaves have the same depth.
  
  Via path length, a phylogenetic tree $\t$ turns $X$ into a metric space, with metric $d_\t\colon X \times X \to \RR$.
  If $\t$ is equidistant, then $d_\t$ is an \emph{ultrametric}, i.e., if for all pairwise distinct $a,b,c \in X$, the maximum 
  \begin{align}\label{eq:three-point}
    \max\{d_\t(a,b), d_\t(a,c), d_\t(b,c)\}
  \end{align} 
  is attained at least twice.
  Conversely, every ultrametric on $X$ corresponds uniquely to an equidistant phylogenetic tree with taxa $X$; see \cite[Theorem 7.2.5]{SempleSteel:2003}.
  
  We identify $d_\t$ with $\t$ and, since the metric $d_\t$ is symmetric, we view both as a point $\t$ in $\RR^n$ for $n = \binom{N}{2}$ and $N \coloneq \abs{X}$.
  \begin{example}
    The tree below has four leaves and has height 3.
    The pairwise distances of the leaves are encoded in the vector on the right:
    \begin{figure}[h]
      \tikzset{baseline={(0,-9mm)},every label/.append style={font=\small}, treenode/.style={fill, inner sep=.75pt, shape=circle}}
      \begin{tikzpicture}
        \path[level distance=5mm, sibling distance=5mm] 
        node[treenode] (root) {}
        child { node[treenode,label={[name=l2]left:}] {}
          child { node[treenode,label={[name=l1]left:}] {}
            child {node[treenode,label=below:$\strut a$] {}}
            child {node[treenode,label=below:$\strut b$] {}}
          }
          child[level distance=10mm, sibling distance=10mm] {node[treenode,label=below:$\strut c$] {}}
        }
        child[level distance=15mm, sibling distance=15mm] {node (d) [treenode,label=below:$\strut d$] {}};
 
        \coordinate (barx) at ($(root.west)+(-15mm,0)$);
        
        \coordinate (top)    at (barx |- root.north);   
        \coordinate (bottom) at (barx |- d.south);      
        
        \draw[line width=0.5pt] (top) -- (bottom);
        
        \draw ($(top)+( -1mm,0)$) -- ++(2mm,0) node[label=left:$3$] {};
        \draw ($(top)+( -1mm,-5mm)$) -- ++(2mm,0) node[label=left:$2$] {};
        \draw ($(bottom)+( -1mm,+5mm)$) -- ++(2mm,0) node[label=left:$1$] {};
        \draw ($(bottom)+(-1mm,0)$) -- ++(2mm,0);
        
        \coordinate (arrowstart) at ($(d.east)+(10mm,7.5mm)$);
        \draw[|->, shorten >=4mm] (arrowstart) -- ++(10mm,0mm) node[right] {\small $\begin{pmatrix} d_\t(a,b) \\ d_\t(a,c) \\ d_\t(a,d) \\ d_\t(b,c) \\ d_\t(b,d) \\ d_\t(c,d) \end{pmatrix} \ = \ \begin{pmatrix} 2 \\ 4 \\ 6 \\ 4 \\ 6 \\ 6 \end{pmatrix}$.};
      \end{tikzpicture}
    \end{figure}
  \end{example}
  We define the following equivalence relation $\sim$ on the set $\mathcal{T}_X$ of equidistant phylogenetic trees with taxa $X$:
  for $\t, \t' \in \mathcal{T}_X$, we let $\t \sim \t'$ if $\t$ is obtained from $\t'$ by lengthening (or shortening) the paths from every leaf of $\t$ to its respective parent node by the same amount.
  Viewing $\t$ and $\t'$ as elements of $\RR^{n}$, we have $\t \sim \t'$ if and only if there exists $\lambda \in \RR$ such that $\t = \t' + \lambda\one$, where $\one$ is the all ones vector.
  We may view thus the set of equidistant phylogenetic trees up to $\mathcal{T}_X/{\sim}$ as a subset of 
  \[
    \torus{n} \ \coloneqq \ \Set{\TreeFnt{x} + \RR\one;\TreeFnt{x} \in \RR^n} \ .
  \]
  The latter is called \emph{$n${-}dimensional tropical torus}.
  We will see in the following section that both the tropical torus and $\mathcal{T}_X/{\sim}$ appear naturally in tropical geometry.
  
  \subsection{Tropical perspective}
  Before we elaborate on the tropical geometric nature of $\mathcal{T}_X$ and $\mathcal{T}_X/{\sim}$ we introduce some basic notions.
  For a more detailed introduction to tropical geometry see \cite{MaclaganSturmfels} or \cite{ETC}.
  
  Often, tropical geometry is termed the combinatorial shadow of algebraic geometry. 
  Like in algebraic geometry we can define \emph{tropical hypersurfaces}, \emph{tropical curves} and \emph{tropical linear spaces}.
  These objects live in a tropical (sub-)space.  
  One dimensional tropical space, namely the \emph{tropical semiring} $(\TT, \oplus, \odot)$ with ground set $\TT \coloneq \RR \cup \{-\infty\}$, replaces classical addition $+$ with $\oplus \coloneq \max$ and classical multiplication $\cdot$ with $\odot \coloneq +$. 
  Hence, $-\infty$ becomes the neutral element with respect to tropical addition $\oplus$ and $0$ the neutral element with respect to tropical multiplication $\odot$.
  Note, that there exists no unique additive inverse element.
  The monoid $\TT^n$ with componentwise tropical addition becomes a semimodule over $\TT$ via tropical scalar multiplication 
  \[
    \odot: \TT \times \TT^n \to \TT^n, \ (\lambda, v) \mapsto \lambda \odot v \coloneq (\lambda + v_1, \ldots, \lambda + v_n) = \lambda\one + v \ .
  \]

  Ardila and Klivans \cite[Theorem 3]{ArdilaKlivans} showed that $\mathcal{T}_X$ coincides exactly with a tropical linear space, namely the 
  the Bergman fan $\mathcal{B}(K_N) \subseteq \RR^n$ of the complete graphical matroid on $N$ elements.
  We omit its precise definition as it will not be relevant and would require a more extensive introduction to matroids in particular which is beyond the scope of what will follow. 
  
  The Bergman fan $\mathcal{B}(K_N)$ is a polyhedral complex, in particular, as the name suggests, a polyhedral fan.
  One often differentiates between two different fan structures of the Bergman fan, known as the \emph{coarse} and \emph{fine subdivision}.
  For the coarse subdivision, there exists a bijection between its cones and combinatorial tree types; see \cite[Proposition 3]{ArdilaKlivans}.  
  While known before, in \cite[Proposition 2]{Hampe:2015} Hampe gave a proof that $\mathcal{B}(K_N)$ is \emph{tropically convex},
  meaning that for $x, y \in \mathcal{B}(K_N)$ and $\lambda, \mu \in \RR$, the tropical linear combination
  \begin{equation}\label{eq:tropical-convexity}
    \lambda \odot x \oplus \mu \odot y \ = \ \max\{x + \lambda \one, y + \mu \one\}
  \end{equation}
  is again an element of $\mathcal{B}(K_N)$.
  This implies that $\mathcal{B}(K_N)$ has a \emph{lineality space} of dimension $\geq 1$ and contains the line spanned by $\one$, i.e., 
  for any point $x \in \mathcal{B}(K_N)$ the whole line $x + \RR\one$ is contained in $\mathcal{B}(K_N)$.
  Therefore its quotient $\bergman{N} \coloneq \mathcal{B}(K_N)/\RR\one$ is well defined and remains tropically convex.
  
  \begin{remark}
    The motivation behind the tropical torus $\torus{n}$ comes from tropical convexity. 
    In fact any set $C \subseteq \RR^n$ satisfying \zcref{eq:tropical-convexity} for any $x, y \in C$ contains the line spanned by $\one$ in its lineality space.
    Hence we always consider tropically convex sets in the tropical torus $\torus{n}$.
  \end{remark}
  
  In the following we will only consider equidistant trees.
  Mostly, we will be considering equivalence classes $t = \t + \RR\one \in \bergman{N} \subseteq \torus{n}$.
  For brevity, we also refer to $\bergman{N}$ as the Bergman fan. 
  
  \subsection{Tropical median consensus tree}\label{sec:consensus}
  A consensus tree of a set of phylogenetic trees $\mathcal{S} \subseteq \torus{n}$ aims to best represent all trees in $\mathcal{S}$.
  One candidate, the \emph{tropical median consensus tree}, was introduced by \textcite[Definition 20]{ComaneciJoswig:24}; see \zcref{def:tropical-median-consensus} below.
  It is defined in terms of the asymmetric tropical distance as follows.
  The quasi-metric $\tropdist$ on $\RR^n$ from \eqref{eq:def-tropdist} is invariant under tropical scalar multiplication;
  i.e., $\tropdist(x + \lambda \mathbb{1}, y + \mu \mathbb{1}) = \tropdist(x, y)$ for every $\lambda, \mu \in \RR$.
  Therefore, it induces a quasi-metric $\tropdist$ on the tropical torus $\torus{n}$
  which we call the \emph{asymmetric tropical distance on $\torus{n}$}.
  
  The $\tropdist$-unit ball
  \[
    \Set{x \in \torus; \tropdist(x,0) \leq 1}
  \]
  in $\torus{n}$ is the standard simplex $\Delta + \RR\one = \conv\{e_1 + \RR\one, \ldots, e_n + \RR\one\}$,
  where the $e_i$ denote the standard basis vectors of $\RR^n$.
  For $x,y \in \torus{n}$, we have that
  \begin{equation}
  	\label{eq:tropdist-ball}
  	\tropdist(x,y) = \min\Set{\lambda \in \RR_{\geq 0}; y \in x + \lambda(\Delta+\RR\one)} \ .
  \end{equation}
  The \enquote{amount of asymmetry} of $\tropdist$ is measured by its \emph{skewness}
  \begin{equation}\label{eq:skewness}
    \sigma_n \ = \ \min\Set[\big]{ \lambda \geq 1 ; \tropdist(x,y) \leq \lambda \cdot \tropdist(y,x) , \ \forall x,y \in \torus{n},\, x \neq y} \ ;
  \end{equation}
  see \cite{Plastria}.
  Clearly, $\Delta + \RR\one$ is compact.
  Hence we can apply \cite[Lemma 10]{Plastria} and obtain
  \begin{equation}\label{def:skewness}
    \sigma_n = \max\left\{\tropdist(x,0) \mid x \in \Delta + \RR\one\right\} = n-1 \ .
  \end{equation}

  \begin{remark}
    The skewness of the tropical asymmetric distance restricted to the Bergman fan $\bergman{N}$ remains $n-1$.
  	Too see this, consider for example $t = 0$ and $t' = -e_1$. 
    Since both satisfy \eqref{eq:three-point} and are therefore contained in $\bergman{N}$,
 	  we have $\tropdist(t',t) = n-1$.
    \zcref[S]{fig:bergmanfan} demonstrates this for $N=4$.
    With the trivial tree $0$ at the center, the figure displays an excerpt of $\bergman{4}$ 
    with the isolines $\tropdist(0, -) = 1,2,\dotsc$ in black and isolines $\tropdist(-, 0) = 1,2\dotsc$ in red. 
    The blue points correspond to trees at distance $1$ to $0$ and distance $5$ from $0$.
  \end{remark}
  
  \begin{figure}[t]
    \newsavebox\TreeTypeA
    \newsavebox\TreeTypeB
    \newsavebox\TreeTypeC
    \newsavebox\TreeTypeD
    \newsavebox\TreeTypeE
    \newsavebox\TreeTypeF
    \tikzset{T/.style={nodes={fill, inner sep=.5pt, shape=circle}}}
    \savebox\TreeTypeA{
      \tikz[T, level distance=3mm, sibling distance=1mm]
      \path
      node (root) {}
      child[level distance=2mm, sibling distance=1.5mm] {[level distance=1mm, sibling distance=1mm]
        node {}
        child {node {}}
        child {node {}}
      }
      child { node {} }
      child { node {} };
    }
    \savebox\TreeTypeB{
      \tikz[T, level distance=3mm, sibling distance=1mm]
      \path
      node (root) {}
      child[level distance=2mm, sibling distance=3mm] {[level distance=1mm, sibling distance=1mm]
        node {}
        child {node {}}
        child {node {}}
        child { node {}}
      }
      child { node {} };
    }
    \savebox\TreeTypeC{
      \tikz[T, level distance=3mm, sibling distance=1mm]
      \path
      node (root) {}
      child[sibling distance=1.5mm] {node {}}
      child[level distance=2mm] {[level distance=1mm, sibling distance=1mm] node {}
        child {node {}}
        child {node {}}
      }
      child[sibling distance=1.5mm] { node {} };
    }
    \savebox\TreeTypeD{
      \tikz[T, level distance=3mm, sibling distance=1mm]
      \path
      node (root) {}
      child { node {} }
      child[level distance=2mm, sibling distance=3mm] {[level distance=1mm, sibling distance=1mm]
        node {}
        child {node {}}
        child {node {}}
        child { node {}}
      };
    }
    \savebox\TreeTypeE{
      \tikz[T, level distance=3mm, sibling distance=1mm]
      \path
      node (root) {}
      child { node {} }
      child { node {} }
      child[level distance=2mm, sibling distance=1.5mm] {[level distance=1mm, sibling distance=1mm]
        node {}
        child {node {}}
        child {node {}}
      };
    }
    \savebox\TreeTypeF{
      \tikz[T, level distance=3mm, sibling distance=1mm]
      \path
      node (root) {}
      child { node {} }
      child { node {} }
      child { node {} }
      child { node {} };
    }
    \begin{tikzpicture}
      \node (I) {\includegraphics[width=4cm]{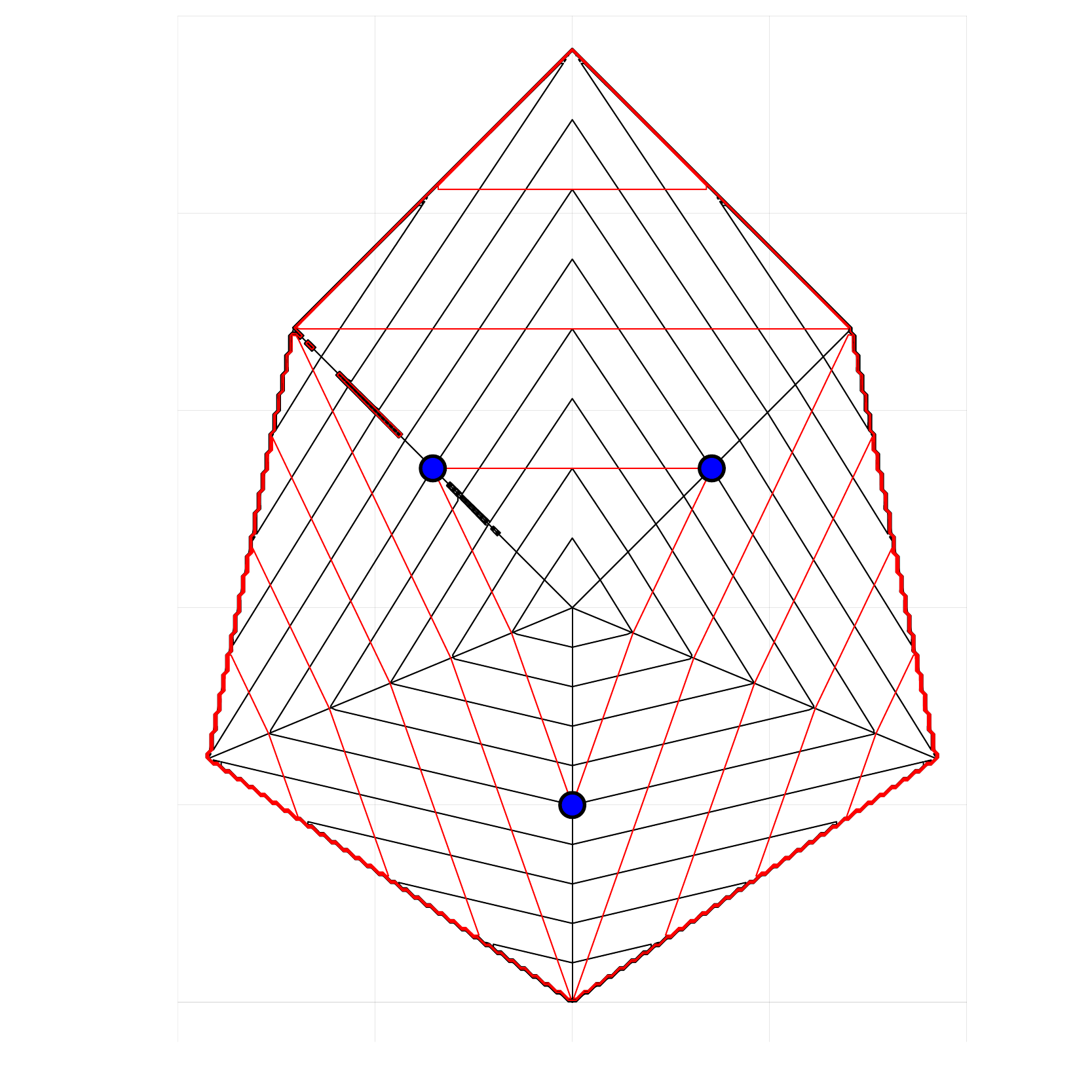}};
      \coordinate (c) at ([yshift=-3mm]$(I.south west)!.5!(I.north east)$);
      \node at ($(c)+(45:2)$) {\usebox\TreeTypeA};
      \node at ($(c)+(-22.5:2)$) {\usebox\TreeTypeB};
      \node at ($(c)+(-90:1.9)$) {\usebox\TreeTypeC};
      \node at ($(c)+(-157.5:2)$) {\usebox\TreeTypeD};
      \node at ($(c)+(135.5:2)$) {\usebox\TreeTypeE};
    \end{tikzpicture}
    \caption{%
      Excerpt from the Bergman fan $\bergman{4}$.
      The excerpt corresponds to the trees that have taxa $a$, $b$, $c$, $d$ in this order.
      Each cell corresponds to one combinatorial type of such a tree.
      The edges that join the cells correspond to degenerate trees as shown.
      The black lines are isolines of $\tropdist(0, -)$, where $0$ is the tree \usebox\TreeTypeF corresponding to the center,
      and the red lines are the isolines of $\tropdist(-, 0)$.
    }
    \label{fig:bergmanfan}
  \end{figure}
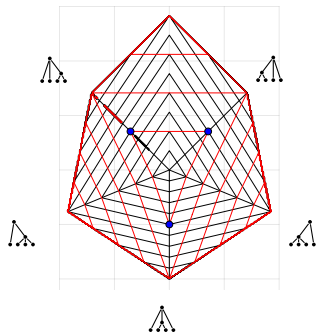
  
  We conclude with the following inequality.
  
  \begin{lemma}[Pseudo-triangle inequality]\label{lem:quasi-triangle}
    Let $x,y,z \in \torus{n}$, then 
    \[
      \frac{1}{n-1}\tropdist(x, y) \ \leq \ \frac{1}{n-1}\tropdist(x, z) + \tropdist(y, z) \ .
    \]
  \end{lemma}
  \begin{proof}
    We have the triangle inequality
    \[
      \frac{1}{n-1}\tropdist(x, y) \ \leq \ \frac{1}{n-1}\tropdist(x, z) + \frac{1}{n-1}\tropdist(z, y) \ .
    \]
    Using that $\tropdist$ has skewness $n-1$, we obtain
    \begin{equation*}
     	\frac{1}{n-1}\tropdist(x, y) + \frac{1}{n-1}\tropdist(z, y) \ \leq \ \frac{1}{n-1}\tropdist(x, z) + \tropdist(y, z).
      \qedhere
    \end{equation*}
  \end{proof}

  Next we give and explicit description of how the tropical median consensus tree $t$ of a set of equidistant trees $\mathcal{S}$ is computed. 
  First, we note that $t$ minimizes the total distance
  \begin{equation}\label{eq:fw-set}
    \sum_{s \in \mathcal{S}} \tropdist(s,t)
  \end{equation}
  from the trees $s \in \mathcal{S}$ to $t$.
  This characterization alone does not render $t$ unique.
  In fact, the dimension of the set of points which minimize \zcref{eq:fw-set}, called the \emph{Fermat--Weber set} $\FW(\mathcal{S})$, is bounded by 
  \begin{equation}\label{eq:dim_trop_median}
    \min\Set[\big]{N-1, \gcd(\abs{\mathcal{S}},n)} - 1 \ .
  \end{equation}
  It can be computed with a linear program in polynomial time \cite[Corollary 16]{ComaneciJoswig:24}. 
  Geometrically, the Fermat--Weber set $\FW(\mathcal{S})$ is a \emph{polytrope} \cite[Theorem 4]{ComaneciJoswig:24},
  i.e., a classical polytope which in addition is also tropically convex.
  
  Like a classical convex hull, the Fermat--Weber polytrope $\FW(\mathcal{S})$ has a minimal generating set $v_1, \dotsc, v_l$, called \emph{tropical vertices}, 
  such that any point in $\FW(\mathcal{S})$ is a tropical linear combination of $v_1, \dotsc, v_l$.
  This allows us to give a unique definition of the tropical median consensus tree. 
  \begin{definition}\label{def:tropical-median-consensus}
    The \emph{tropical median consensus tree} of $\mathcal{S}$ is the ordinary, componentwise mean of the tropical vertices of $\FW(\mathcal{S})$.
    More explicitly, it is the tree
    \[
      \frac{1}{l}\sum_{i=1}^{l} v_i \ .
    \]
  \end{definition}
  
  \begin{remark}\label{rm:trop_convex_median}
    Com\u aneci and Joswig showed that any point in the Fermat--Weber set $\FW(\mathcal{S})$ can be written as a tropical linear combination of the points in $\mathcal{S}$; see \cite[Theorem 4]{ComaneciJoswig:24}.
    This is particularly true for the tropical median consensus tree of $\mathcal{S}$.
  \end{remark}
  
  \begin{remark}
    There does exist a metric on the tropical torus $\torus{n}$, called the \emph{symmetric tropical distance} 
    \begin{equation}
      \label{eq:sym_tropdist}
      \symdist(x, y) \ = \ \max_{i  \in [n]}\{x_i - y_i\} - \min_{j  \in [n]}\{x_j - y_j\} \ .
    \end{equation}
    Although one might be inclined towards the symmetric distance, the asymmetric distance has certain theoretical and practical advantages.
    Most notably, contrary to the symmetric distance, there exists a consensus method for equidistant phylogenetic trees. 
    
    Note that another way to determine the skewness \zcref{def:skewness} of the tropical distance is to compare it with the symmetric tropical distance \eqref{eq:sym_tropdist}; see \cite[Section 2]{ComaneciJoswig:23}.    
  \end{remark}
  
  We now have all the necessary tools required to define tropical $k$-means clustering.
  
  \section{tropical \K-means++ clustering}\label{sec:tropical-k-means}
  In this section we will introduce the \emph{tropical $k$-means++} algorithm.
  We start by recalling the $k$-means clustering.
  
  For a subset $S \subseteq V$ of \emph{sites} in a quasi-metric space $(V,d)$,
  $k$-means clustering aims at partitioning $S$ into $k$ disjoint subsets $C_1,\dotsc,C_k$, called \emph{clusters},
  by finding $k$ points $c_1,\dotsc,c_k \in V$, called centroids, that minimize the objective function
  \begin{equation}
  	\label{eq:clustering-objective-func}
  	\sum_{s \in S} \min_{i\in[k]} d(s, c_i)^q.
  \end{equation}
  The clusters are the \emph{Voronoi regions} $C_i = \Vor_{c_i}(S) = \Set{s \in S; d(s, c_i)^q \leq d(s, c_j)^q \ \forall \ j}$.
  
  The problem cannot be solved analytically in practice.
  By using the following iterative algorithm due to Lloyd~\cite{Lloyd:1982}, it is possible to find a local optimum:
  choose initial values for $c_1,\dotsc,c_k \in S$, and run the following steps in turns:
  \begin{enumerate}[label=\arabic*.]
  	\item Set $C_i \coloneqq \Vor_{c_i}(S)$. We call this the \emph{A-step}.
  	\item Set $c_i \coloneqq \argmin_{c \in S} \sum_{s \in C_i} d(s, c)^q$. We call this phase the \emph{M-step}.
  \end{enumerate}
  The algorithm terminates if either step does not change any $C_i$ or $c_i$, respectively.
  If $S$ is finite, the above iterative algorithm is easily seen to terminate,
  though, possibly, in a local optimum.
  
  
  \subsection{Tropical clustering}
  For us, the sites we consider correspond to equidistant phylogenetic trees.
  Hence we consider points in the tropical torus $\torus{n}$ and apply the tropical asymmetric distance $\tropdist$.
  
  \begin{algorithm}
    \caption{tropical $k$-means clustering}
    \label{al:kmedian}
    \begin{algorithmic}[1]
      \REQUIRE Dataset $S \subseteq \torus{n}$ of finitely many points, number of clusters $k$
      \ENSURE Cluster assignments for each $s$ and centroids $c_1, c_2, \dots, c_k$
      
      \STATE Initialize centroids $c_1, c_2, \dots, c_k$ (e.g., randomly select $k$ points from $S$)
      \REPEAT
      \FOR{each $s \in S$}
      \STATE Assign $s$ to cluster $C_j$ with closest centroid:\\
      \hspace{0.5cm} $C_j \gets \Set{s \in S; j = \arg\min_{i \in [k]} \tropdist(s, c_i)}$
      \ENDFOR
      \FOR{each cluster $C_j$}
      \STATE Update centroids:\\
      \hspace{0.5cm} $c_j \gets \textrm{tropical median of }C_j$
      \ENDFOR
      \UNTIL{cluster assignments do not change}
    \end{algorithmic}
  \end{algorithm}
  
  Analogously to \cite[\S 3.6]{MacQueen:1967}, one can show that \zcref{al:kmedian} terminates after finitely many steps.
  For completeness we include a proof.
  
  \begin{proposition}\label{prop:convergence}
    The tropical $k$-means clustering algorithm terminates.
  \end{proposition}
    
  \begin{proof}
  	Let $\mathcal{C} = \{C_1,\dotsc,C_k\} \subseteq 2^S$ be the set of clusters after the M-step of an arbitrary iteration,
  	and $c_i$ be the tropical median of $C_i$ for $i = 1,\dotsc,k$ be the respective tropical medians.
  	During the A-step, we choose the new clusters $C'_i = \Vor_{c_i}(S)$.
  	For $s \in S$, let $\gamma_s$ be the integer such that $s \in C_{\gamma_s}$,
  	and define $\gamma'_s$ analogously.
  	By definition of the Voronoi region, we have $\gamma'_s = \argmin_i \tropdist(s, c_i)$,
  	so for all $s$, we have $\tropdist(s, c_{\gamma'_i}) \leq \tropdist(s, c_j)$ for all $j$.
  	In particular, we get
    \begin{equation}
      \label{eq:loss-decrease-1}
      \ell(\mathcal{C}') \ = \ \sum_{s \in S} \tropdist(s, c_{\gamma'_s}) \ \leq\ \sum_{s \in S} \tropdist(s, c_{\gamma_s})\ = \ \ell(\mathcal{C})\ .
    \end{equation}
    Therefore, the total loss does not increase during the A-step,
    with equality if and only if $\mathcal{C}' = \mathcal{C}$.
    
    Second, the M-step replaces the centroids by the tropical median of $C_i$, denoted $c'_i$.
    Since the tropical median is a Fermat--Weber point, it satisfies $c'_i \in \argmin_c \sum_{s \in C'_i}\tropdist(s, c)$.
    Therefore, we obtain
    \begin{equation}
    	\sum_{s \in S} \tropdist(s, c'_{\gamma'_s}) \ \leq\ \sum_{s \in S} \tropdist(s, c_{\gamma'_s})\ \overset{\eqref{eq:loss-decrease-1}}{\leq} \ \sum_{s \in S} \tropdist(s, c_{\gamma_s})\ ,
    \end{equation}
    with equality if and only if $c'_i = c_i$ for all $i$, in which case the algorithm terminates.
    Since $S$ is finite, the algorithm terminates in finitely many steps.
  \end{proof}
  
  Like its Euclidean counterpart, the algorithm can be stuck in a local optimum as the following example demonstrates.
  \begin{example}\label{ex:local-optimum}
    Consider the points
    \begin{align*}
      v_1 &= \begin{psmallmatrix}
        1 \\ 3 \\ 0
      \end{psmallmatrix},& \ v_2 &= \begin{psmallmatrix}
        0 \\ 3 \\ 1
      \end{psmallmatrix},& \ v_3 &= \begin{psmallmatrix}
        0 \\ 0 \\ 1
      \end{psmallmatrix},& \ v_4 &= \begin{psmallmatrix}
        1 \\ 0 \\ 0
      \end{psmallmatrix};
    \end{align*}
		see Fig.~\ref{fig:local-optimum}.
    With $c_1 = v_1$ and $c_2 = v_2$ as initial centroids, the algorithm will terminate after the first iteration
    with clusters $C_1 = \{v_1, v_4\}$ and $C_2 = \{v_2,v_3\}$.
    The loss of this clustering $\mathcal{C} = \{C_1, C_2\}$ is $\ell(\mathcal{C}) = 6$.
    An optimal clustering $\mathcal{C}^* = \{C_1^*, C_2^*\}$ instead is $C_1^* = \{v_1,v_2\}$ and $C_2^* = \{v_3,v_4\}$.
    Their respective tropical medians are the points 
    \begin{align*}
      c^*_1 &= \begin{psmallmatrix}
        0 \\ 2 \\ 0
      \end{psmallmatrix},&
      c^*_2 &= \begin{psmallmatrix}
        1 \\ 0 \\ 1
      \end{psmallmatrix},
    \end{align*}
    resulting in $\ell(\mathcal{C}^*) = 4$.
    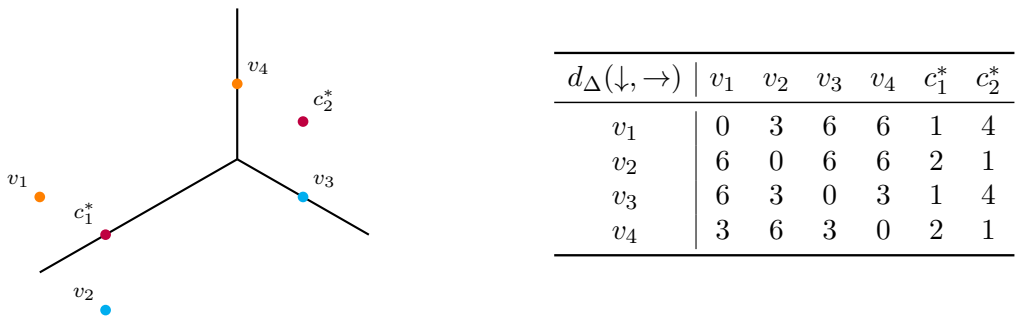
\begin{figure}
      \begin{tikzpicture}[scale=1, baseline={(current bounding box.center)}, nodes={font=\scriptsize}]
        \coordinate (e1) at (0,1);                       
        \coordinate (e2) at ({cos(210)},{sin(210)});     
        \coordinate (e3) at ({cos(330)},{sin(330)});
        
        \draw[thick] (0,0) -- ($2*(e1)$);
        \draw[thick] (0,0) -- ($3*(e2)$);
        \draw[thick] (0,0) -- ($2*(e3)$);
          
        \coordinate (v1) at ($(e1)+3*(e2)$);
        \coordinate (v2) at ($3*(e2)+(e3)$);
        \coordinate (v3) at (e3);
        \coordinate (v4) at (e1);
          
        \foreach \p/\pos in {1/above left,4/above right}
          \fill[fill=orange] (v\p) circle (2pt) node[\pos] {$v_{\p}$};
            
        \foreach \p/\pos in {2/above left,3/above right}
          \fill[fill=cyan] (v\p) circle (2pt) node[\pos] {$v_{\p}$};
            
        \coordinate (a) at ($2*(e2)$);
        \coordinate (b) at ($-1*(e2)$);
          
        \foreach \p/\l/\pos in {a/c^*_1/above left,b/c^*_2/above right}
          \fill[fill=purple] (\p) circle (2pt) node[\pos] {$\l$};
      \end{tikzpicture}
      \hfill
      $\begin{array}{c|cccccc}
      	\toprule
      	d_\Delta(\downarrow,\rightarrow) & v_1 & v_2 & v_3 & v_4 & c^*_1 & c^*_2 \\
      	\midrule
      	v_1                              & 0   & 3   & 6   & 6   & 1     & 4     \\
      	v_2                              & 6   & 0   & 6   & 6   & 2     & 1     \\
      	v_3                              & 6   & 3   & 0   & 3   & 1     & 4     \\
      	v_4                              & 3   & 6   & 3   & 0   & 2     & 1     \\
      	\bottomrule
      \end{array}$
      \caption{A local optimum for the tropical $2$-median clustering, with centroids $c_1 = v_1$ and $c_2 = v_2$. The points $c^*_1$ and $c^*_2$ correspond to the optimal clustering.}
      \label{fig:local-optimum}
    \end{figure}
  \end{example}
  
  However it is possible to choose the initial centroids such that the loss of a local optimum is bounded by the loss of the global optimum.
  Before we discuss this modification we want to investigate the relation of classical and tropical $k$-means algorithm in more depth.
  In fact tropical $k$-means does not arise as a tropicalized version of classical $k$-means algorithm in the Euclidean setting as the following example shows.
  \begin{example}
    This is a modified version of \cite[Example 1]{ComaneciJoswig:24}.
    Consider the points 
    \begin{align*}
      V_{\alpha} = \begin{pmatrix}
        14 & 13 & 11 - \frac{\alpha}{2} & 10 & 3 & 16\\
        -7 & -14 & -13 - \frac{\alpha}{2} & 1 & -3 & -1\\
        -7 & 1 & 2 + \alpha & -11 & 0 & -15
      \end{pmatrix}
    \end{align*}
    Let $v_1,\ldots, v_6$ denote their columns.
    For any $\alpha \geq 2$ the tropical median of $v_1,\ldots,v_5$ is $(9,-6,-3)^{\top}$.
    
    We cluster the set of points with tropical $k$-means clustering into two sets, i.e, $k = 2$.
    Choose $v_4$ and $v_6$ as initial centroids.  
    After the first A-step the points will be partitioned into the sets $\{v_1,v_2,v_3,v_4,v_5\}$ and $\{v_6\}$ for any $\alpha \geq 2$.
    Note that after replacing the initial centroids by the respective tropical median the partitioning remains unchanged after the second A-step.
    Thus \zcref{al:kmedian} terminates.
    
    In contrast, the tropicalized version of classical $k$-means replaces the tropical median by the point-wise maximum.
    We pick the same initial centroids as before.
    While the sets will be again partitioned into $\{v_1,v_2,v_3,v_4,v_5\}$ and $\{v_6\}$ after the first A-step,
    note that the point-wise maximum is highly sensitive to outliers.
    In fact, after the second assignment step the sets will be partitioned into $\{v_3\}$ and $\{v_1,v_2,v_4,v_5,v_6\}$.
    After replacing the centroids with the point-wise maximum the partition remains unchanged after the third A-step.
  \end{example}
  
  \subsection{\K-means++ clustering}
  As shown above, the local optimum $\mathcal{C}'$ Lloyd's algorithm converges to depends on the initial choice for the centroids.
  Additionally, the quotient $\ell(\mathcal{C}')/\ell(\mathcal{C}^*)$ of the loss of $\mathcal{C}'$ and of the global optimum $\mathcal{C}^*$ can be arbitrarily large.
  The \emph{$k$-means++ algorithm} \cite{ArthurVassilvitskii:2007} bounds the expectation value of $\ell(\mathcal{C}')/\ell(\mathcal{C}^*)$
  by combining Lloyd's algorithm with the following initial seeding for the centroids:
  choose $c_1 \in S$ uniformly at random, and for $j > 1$, choose $c_j \in S$ at random with probability $p(s) = \frac{\min_{i < j} d(s, c_i)}{\sum_{t \in S} \min_{i < j} d(t, c_i)}$.
  Then proceed with $k$-means clustering as above.
  With this seeding, one can bound, roughly speaking, how far from the global optimum a local optimum the algorithm converges to can be.
  To make that precise, we define the \emph{competitive factor} of a clustering $\mathcal{C}$ as $\frac{\ell(\mathcal{C})}{\ell(\mathcal{C}^*)}$,
  where $\mathcal{C}^*$ is the clustering for which the global optimum of \eqref{eq:clustering-objective-func} is attained.
  
  \begin{remark}
  	The competitive factor of Euclidean $k$-means++ algorithm is $8(2 + \log k)$ \cite{NielsenSun:2019}.
  \end{remark}
  
  The competitive factor of tropical $k$-means clustering follows directly from \cite[Theorem~2]{NielsenSun:2019} and \eqref{eq:skewness}.
  
  \begin{proposition}\label{thm:competitive}
    For a finite set of sites $S \subseteq \torus{n}$, the expectation value of the competitive factor of the clustering obtained by tropical $k$-means++ algorithm is at most $2n(2 + \log k)$.
  \end{proposition}
  Next, we aim to establish when tropical $k$-means++ is able to recover a partition of a set of equidistant trees.
  
  \subsection{Coarse type}
    We will show that tropical $k$-means++ clustering partitions a set of trees by combinatorial type,
    assuming the types are ``sufficiently pronounced''.
    However, a large heterogeneous data set of phylogenetic trees may contain trees of many different combinatorial types.
    This leads to the following coarsening of the combinatorial type.
  
%
%
  
  \begin{definition}
    For a tree $t$, we respectively write $\nu(t)$ and $\eta(t)$ for the minimal and maximal depth of the inner nodes of $t$.
    The \emph{coarse type} of a phylogenetic tree $t$ with taxa $X$
    is the unique coarsest partitioning $\mathcal{X} \subseteq 2^X$ of $X$ such that for all $X' \in \mathcal{X}$, all $a \in X'$ have a common ancestor of depth at least $\nu(t)$.
  \end{definition}
  
  We want to take a closer look at the set of trees $\Gamma \subseteq \bergman{N}$ with a given coarse type $\mathcal{X}$. 
  Note that an equidistant tree $t$ has coarse type $\mathcal{X} = \{X_i \mid i \in I\}$ if and only if there exits $h \in \RR$ such that
  \begin{enumerate}[label=\arabic*.,leftmargin=*]
    \item the distances of two leaves $a, b \in X_i$ for $i \in I$ is smaller than $2h$ and
    \item the distance of two leaves $a \in X_i$, $b \in X_j$ for $i,j \in I, i \neq j$ is equal to $2h$.
  \end{enumerate} 
  This gives us another way to write the set $\Gamma$. 
  Namely, $\Gamma$ is the set of trees $t \in \bergman{N}$ which satisfy 
  \begin{multline}\label{eq:coarsetopology}
    \max\{t_{ab} \mid a, b \in X_i\} <\min\{t_{ab} \mid a \in X_i, b \in X_j, i \neq j\} \\
    = \max\{t_{ab} \mid a \in X_i, b \in X_j, i \neq j\} \ .
  \end{multline} 
  
  \begin{proposition}
    The set $\Gamma$ is tropically convex.
  \end{proposition}
  
  \begin{proof}
    First, we note that the Bergman fan $\mathcal{B}(K_{N})$ is tropically convex; see \cite[Proposition 2.15]{Hampe:2015}. 
    The set of trees $\Gamma$ with same coarse type is thus tropically convex if any tropical convex combination of two trees $s,t \in \Gamma$ satisfies \zcref{eq:coarsetopology}.
    Let $\lambda, \mu \in \RR$ and 
    \[
    z \ \coloneq \ \max\{s + \lambda \one, t + \mu \one\}
    \]
    be a tropical convex combination of $s$ and $t$.
    It follows immediately that
    \begin{align*}
      \min\{z_{ab} \mid a \in X_i, b \in X_j, i \neq j\} = \max\{z_{ab} \mid a \in X_i, b \in X_j, i \neq j\} \ .
    \end{align*}
    Next there exists $\tilde{a}, \tilde{b} \in X_k$ for some $k \in I$ such that 
    \[
    \max\{z_{ab} \mid a, b \in X_i\} = z_{\tilde{a}\tilde{b}} = \max\{s_{\tilde{a}\tilde{b}}+\lambda, t_{\tilde{a}\tilde{b}}+\mu\} \ .
    \]
    Clearly, we have
    \[
    z_{\tilde{a}\tilde{b}} = \max\{s_{\tilde{a}\tilde{b}}+\lambda, t_{\tilde{a}\tilde{b}}+\mu\} < \max\{s_{ab} + \lambda, t_{ab} + \mu \} = z_{ab} \ .
    \] 
    for every $a \in X_i$ and $b \in X_j$ with $i \neq j$.
    This implies that $z$ is contained in $\Gamma$ which shows the statement.
  \end{proof}
  
  Together with \zcref{rm:trop_convex_median} the previous result implies that this consensus tree of $\mathcal{S}$ has the same coarse type if all trees in $\mathcal{S}$ have the same coarse type.
  
  \begin{corollary}
    \label{thm:coarse-type-median}
    Let $\mathcal{S}$ be a set of equidistant trees with the same coarse type.
    The tropical consensus tree of $\mathcal{S}$ has the same coarse type as the trees in $\mathcal{S}$.
  \end{corollary}
  
  Additionally, we can bound the distance of two trees of the same coarse type as follows.
  
  \begin{lemma}\label{lem:upper-bound}
    Let $t, t' \in \Gamma$ two trees with same coarse type $\{X_i \mid i \in I\}$. 
    Additionally let $\Omega > \omega > 0$ such that
    \[
    \max\{\eta(t),\eta(t')\}\leq\Omega \quad \text{and} \quad \min\{\nu(t), \nu(t')\}\geq\omega \ .
    \]
    Then the tropical asymmetric distance of $t$ and $t'$ is 
    \[
    \tropdist(t, t') \leq \left(n + \iota\right)2(\Omega - \omega) 
    \]
    with $\iota \coloneq \sum_{i \in I} \binom{\abs{X_i}}{2} \leq (2n - N + 1)2(\Omega - \omega)$.
    If $t, t'$ are both binary trees then
    \[
    \tropdist(t, t') \leq (2n - N + 1)2(\Omega - \omega) \ .
    \]
  \end{lemma}
  
  \begin{proof}
    Let $h \coloneqq \eta(t)$ and $g \coloneqq \eta(t')$.
    W.l.o.g., we may assume that $t$ has height $h$ and $t'$ has height $g$.
    We begin by showing the first inequality and deduce the second from there.
    
    First, if $a \in X_i$ and $b \in X_j$ with $i \neq j$ then $\abs{t_{ab} - t'_{ab}} = \abs{2h - 2g}$.
    Next, if $a, b \in X_i$, $i \in I$ we have $t'_{ab} - t_{ab} \leq 2(g - \omega)$ and $t_{ab} - t'_{ab} \leq 2(h - \omega)$.
    This allows us to determine $\max\nolimits_{a,b \in X}\{t_{ab} - t'_{ab}\}$.
    We have
    \[
    \max\nolimits_{a,b \in X}\{t_{ab} - t'_{ab}\} \leq \max\{2(h - \omega), 2h - 2g\} = 2h - 2 \omega \ .
    \]
    Putting this into the definition of $\tropdist$ gives
    \begin{align*}
      \tropdist(t, t') &= \sum_{a,b \in X} t'_{ab} - t_{ab} + n \max\nolimits_{a,b \in X}\{t_{ab} - t'_{ab}\} \\
      &= \sum_{a \in X_i, b \in X_j, i \neq j} t'_{ab} - t_{ab} + \sum_{a,b \in X_i, i \in I} t'_{ab} - t_{ab} + n \max\nolimits_{a,b \in X}\{t_{ab} - t'_{ab}\} \\
      &\leq (n - \iota) (2g - 2h) + \iota(2g -2 \omega) + n(2h - 2\omega) \\
      &= n(\underbrace{2g - 2\omega}_{\leq 2\Omega - 2\omega}) + \iota (\underbrace{2h - 2\omega}_{\leq2\Omega- 2\omega}) \leq (n + \iota) 2(\Omega - \omega) \ ,
    \end{align*}
    which shows the first inequality.
    
    If the trees are binary there exists $N' < N$ such that $\iota = \binom{N'}{2} + \binom{N-N'}{2}$.
    As a function in $N'$ on the interval $[1,N-1]$ it attains its maximum in $N' = 1$ and $N' = N-1$.
    Hence second inequality follows since $\iota \leq \binom{N-1}{2} = n - N + 1$.
  \end{proof}
  
  On the other hand, if the two trees $t$ and $t'$ have different coarse type, then $\tropdist(t, t') \geq 2(\eta(t) - \nu(t'))$.
  If the trees are binary, their coarse type ensures that they are further apart.
  
  \begin{lemma}\label{lem:lower-bound}
    Let $t \in \Gamma$ and  $t' \in \Gamma'$ be two binary trees with different coarse type.
    Additionally let $\Omega > \omega > 0$ such that
    \[
    \max\{\eta(t),\eta(t')\}\leq\Omega \quad \text{and} \quad \min\{\nu(t), \nu(t')\}\geq\omega \ .
    \]
    Then the tropical asymmetric distance of $t$ and $t'$ is 
    \[
    \tropdist(t, t') \geq \omega\bigl((3 - 2\tfrac{\Omega}\omega)n + \tfrac{3}{2}N - 2\bigr) \ .
    \]
  \end{lemma}
  \begin{proof}
    Like in the proof of \zcref{lem:upper-bound}
    let $h \coloneqq \eta(t)$ and $g \coloneqq \eta(t')$.
    Again, w.l.o.g., we may assume that $t$ has height $h$ and $t'$ has height $g$.
    
    Each coarse type determines a unique split $\{X_1, X_2\}$ of the set of taxa $X$,
    i.e., a partition of $X$ in two sets.
    Combining the splits of $\Gamma$ and $\Gamma'$, we obtain a partition of $X$ into subsets $Y_1, Y_2, Y', Y''$, where $Y_1, Y_2 \neq \emptyset$, such that $\{Y_1 \cup Y', Y_2 \cup Y''\}$ is the coarse type of $t$ and $\{Y_1 \cup Y'', Y_2 \cup Y'\}$ is the coarse type of $t'$.
    
    Let $N' = \abs{Y'} + \abs{Y''}$.
    Since $i \neq j$, we get $N' \geq 1$.
    The difference between the distances of taxa in $t$ and $t'$ can be estimated from below by
    \begin{equation*}
      t'_{ab} - t_{ab} \ \geq \ \begin{cases}
        2g - 2h           & \text{if $a \in Y_1, b \in Y_2$ or $a \in Y', b \in Y''$},                    \\
        2g - 2h + 2\omega & \text{if $a \in Y_1, b \in Y'$ or $a \in Y_2, b \in Y''$},                    \\
        -2h               & \text{if $a \in Y_1, b \in Y''$ or $a \in Y_2, b \in Y'$},                    \\
        -2h + 2\omega     & \text{if $a,b \in Y_1$ or $a,b \in Y_2$ or $a,b \in Y'$ or $a,b \in Y''$}.
      \end{cases}
    \end{equation*}
    Hence we can compute a lower bound of the distance from $t$ and $t'$ as follows. 
    \begin{align*}
      \tropdist(t, t') \ &= \ \sum\nolimits_{x,y \in \binom{X}{2}} (t'_{ab} - t_{ab}) + n\max\nolimits_{x,y \in \binom{X}{2}}\{t_{ab} - t'_{ab}\} \\
      &\geq\ A(2g - 2h) + (B+n)(2g - 2h + 2\omega) + C(-2h) + D(-2h + 2 \omega)  
    \end{align*}
    with 
    \begin{align*}
      A &\coloneq \abs{Y_1}\abs{Y_2} + \abs{Y'}\abs{Y''}, &
      B &\coloneq \abs{Y_1}\abs{Y'} + \abs{Y_2}\abs{Y''}, \\
      C &\coloneq \abs{Y_1}\abs{Y''} + \abs{Y_2}\abs{Y'}, &
      D &\coloneq \abs{Y_1}^2 + \abs{Y_2}^2 + \abs{Y'}^2 + \abs{Y''}^2.
    \end{align*}
    Note that $A + B + C + D = n$.
    We can rewrite the lower bound in terms of $g, h$ and $\omega$ and obtain
    \begin{align}
      \tropdist(t, t') & \geq 2g(A + B - n) - 2h(A + B + C + D - n) + 2\omega(B + D + n) \notag\\
      & \geq 2g(- n) + 2\omega(A + 2B + D + n)                          \notag\\
      & \geq 2\omega(A + 2B + D + (1-\tfrac\Omega\omega)n)            \notag\\
      & = 2\omega(B - C + (2-\tfrac\Omega\omega)n) \label{eq:lower-bound-N}
    \end{align}
    In the last estimate we applied $g \leq \tfrac\Omega\omega\omega$ again.
    
    In order to find the number $N'$ such that the lower bound \eqref{eq:lower-bound-N} becomes minimal we take a closer look at $B - C$:
    \begin{align}
      B - C & =  \abs{Y_1}\abs{Y'} + \abs{Y_2}\abs{Y''} - \abs{Y_1}\abs{Y''} - \abs{Y_2}\abs{Y'}                                                                                     \notag\\
      & \geq \underbrace{(\abs{Y'} + \abs{Y''})}_{N'} \bigl(\underbrace{\min\{\abs{Y_1}, \abs{Y_2}\}}_{\geq 1} - \underbrace{\max\{\abs{Y_1}, \abs{Y_2}\}}_{\leq N-N'-1}\bigr) \notag\\
      & \geq (N'-N+2)N'                                                                                                                                                        \notag\\
      & \geq -\tfrac14(N-2)^2 \label{eq:ineq-B-C}.
    \end{align}
    Recall that $n = \binom{n}{2}$ and thus $n-N+1 = \binom{N-1}{2}$ or, equivalently, $(N-2)^2 = 2n-3n+4$.
    We conclude the theorem by putting this into \eqref{eq:ineq-B-C} and \eqref{eq:lower-bound-N} and obtain
    \[
    \tropdist(t, t') \ \geq \ \omega \bigl((3 - 2\tfrac{\Omega}\omega)n + \tfrac{3}{2}N - 2\bigr) \ .\qedhere
    \]
  \end{proof}
  
  Finally, our previous work allows us to specify when we are able to recover a given partition of equidistant trees by coarse type.
   
	\begin{theorem}\label{thm:tropicalk-means}
		Let $\mathcal{S}$ be a set of equidistant, binary trees on the same $N$ taxa and $\mathcal{S} = \bigsqcup_{i=1}^k C_i$ be a partition of $\mathcal{S}$,
    such that two trees in $\mathcal{S}$ have the same coarse type if and only if they are contained in the same set $C_i$.
    We assume $N \geq 3$.
    Let $\omega$, $\Omega$ be such that $\nu(t) \geq \omega$ and $\eta(t) \leq \Omega$ for every tree $t \in \mathcal{S}$.
    If
    \begin{equation}
      \label{eq:tropicalk-means-assumption}
      \frac{\Omega}{\omega} \ < \ 1 \: + \: \frac{n+N-1}{4n(2n-N+1)(n-1)(\Omega - \omega) \abs{\mathcal{S}}},
    \end{equation}
    then in expectation, the tropical $k$-means++ clustering algorithm yields the clustering $\mathcal{C} \coloneqq \Set{C_i; i = 1,\dotsc,k}$
    into the original sets $C_i$.
	\end{theorem}
	\begin{proof}
		We prove the theorem follows by establishing the following claims:
		
		\begin{claim}
			The clustering $\mathcal{C}$ has loss $\ell(\mathcal{C}) \leq 2(2n-N+1)(\Omega - \omega) \abs{\mathcal{S}}$.
		\end{claim}
		\begin{claimproof}
			Let $\tropmedian{i}$ denote the tropical median of $C_i$, which has the same coarse type as the trees in $C_i$ by \zcref{thm:coarse-type-median}.
			From \zcref{lem:upper-bound}, we obtain that $\tropdist(t, \tropmedian{i}) \leq 2(2n-N+1)(\Omega - \omega)$ for all $t \in C_i$.
			Hence, we get the total loss 
			\[
				\ell(\mathcal{C}) = \sum_i \sum_{t \in C_i} \tropdist(t, \tropmedian{i}) \leq 2(2n-N+1)(\Omega - \omega) \abs{\mathcal{S}}.
			\]
		\end{claimproof}
		
		\begin{claim}
			\label{claim:other-clustering} Any clustering $\mathcal{C}'$ different than $\mathcal{C}$ has loss greater than $\frac{\omega(n+N-1)}{n-1}$.
		\end{claim}
		\begin{claimproof}
      Let $\tropmedian{C}$ denote the tropical median of a cluster $C \in \mathcal{C}'$.
			From \zcref{lem:quasi-triangle} we can deduce that
			\begin{align*}
			 & \phantom{{}={}} \sum_{C\in \mathcal{C}'}\sum_{t \in C} \tropdist(t, \tropmedian{C})																																				 \\
			 & = \sum_{C\in \mathcal{C}'} \frac{1}{\abs{C}-1} \sum_{\substack{s,t\in C\\t \neq t'}} \bigl(\tropdist(s, \tropmedian{C}) + \tropdist(t, \tropmedian{C}\bigr) \\
			 & \geq \frac{1}{n-1} \sum_{C\in \mathcal{C}'}\frac{1}{\abs{C} - 1}\sum_{s,t\in C} \tropdist(s,t) \ .
			\end{align*}
			Thus, any clustering $\mathcal{C}'$ other than $\mathcal{C}$ has at least one cluster $C$ containing a pair of elements, each from different sets $C_i$, $C_j$.
			One checks that if $C$ contains at least one such pair, then $\binom{C}{2}$ actually contains at least $\abs{C}-1$ such pairs.
			By \zcref{lem:lower-bound}, each of these pairs has tropical distance at least $\omega\bigl((3 - 2\tfrac{\Omega}\omega)n + \tfrac{3}{2}N - 2\bigr)$.
      We show that 
      \[
        \omega\bigl((3 - 2\tfrac{\Omega}\omega)n + \tfrac{3}{2}N - 2\bigr) > \omega(n + N - 1) \ ,
      \]
			or, equivalently, that $\frac{N-2}{4n} > \frac\Omega\omega - 1$.
      The latter inequality follows immediately from $N \geq 3$ since
      \[
      N - 2 \geq 1 > \frac{n - N + 1}{2n(n-1)(2 + \log k)\abs{\mathcal{S}}} \ . 
      \]
      Therefore, the loss of $\mathcal{C}'$ is bounded below by
			\[
				\ell(\mathcal{C}') = \sum_{C\in \mathcal{C}'}\sum_{x \in C} \tropdist(x, \tropmedian{C}) 
				> \frac{\omega(n-N+1)}{n-1} \ .
			\]
		\end{claimproof}
			
		To conclude this proof, recall from \zcref{thm:competitive} that the expectation value of the loss of the clustering obtained by tropical $k$-means++ is at most
		\begin{align}
			2n(2 + \log k)\ell^*
			\ &\leq \ 2n(2 + \log k)\ell(\mathcal{C}) \notag \\
			&= \ 4n(2n-N+1)(2 + \log k)(\Omega - \omega)\abs{\mathcal{S}}\label{eq:exp-value-ineq}\\
			&< \ \frac{\omega(n-N+1)}{n-1} \ . \notag
		\end{align}
		According to \zcref{claim:other-clustering}, any clustering different from $\mathcal{C}$ has loss greater than $\frac{\omega(n+N-1)}{n-1}$.
		Hence, in expectation value, tropical $k$-means++ will not produce a clustering different from $\mathcal{C}$.
	\end{proof}
  
  \begin{example}
  	\begin{figure}
  		\centering
  		\foreach \i in {1,...,4}{%
  			\subcaptionbox*{\i}{%
  				\includegraphics[width=\linewidth/4]{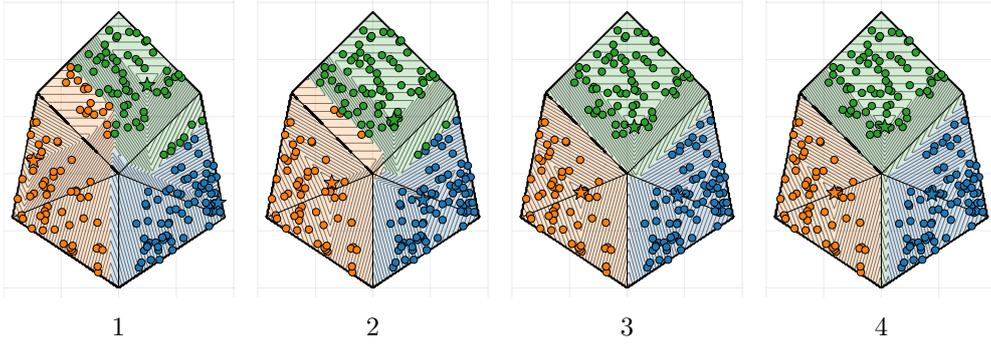}%
  			}%
  		}
  		\caption{%
  			Iterations of tropical $k$-means++ clustering of phylogenetic trees on the four taxa $a$, $b$, $c$, $d$.
  			The visualization of the space of such trees corresponds to \zcref{fig:bergmanfan}.
  		}
  		\label{fig:clustering-steps}
  	\end{figure}
  	We randomly generated 200 phylogenetric trees with taxa $a$, $b$, $c$, $d$ in that order 
    such that $\nu(t) > 0.15$ and $\eta(t) \leq 1$ for every tree $t$.
  	These correspond to points in the excerpt from the Bergman fan $\bergman{4}$ already shown in \zcref{fig:bergmanfan}.
  	The iterations of the clustering are shown in \zcref{fig:clustering-steps}.
  	The stars mark the centroids of the clusters, and the colored regions are the respective Voronoi regions.
  	The centroid in the $(i+1)$st picture is the tropical median of the points in picture $i$ of that color.
  	In the last picture, the process has become stationary.
    The resulting clusters correspond to the 3 different coarse types of the trees.
  	The pictures also show the isolines of the distance to the nearest centroid.
  \end{example}
  
  \section{Oscar Implementation}\label{sec:implementation}
  Our implementation of tropical $k$-means clustering in the computer algebra system \Oscar\ \cite{OSCAR}
  builds upon the implementation of tropical medians \cite{ComaneciJoswig:24} in the software package \Polymake\ \cite{Polymake},
  which is included with \Oscar.
  The method of \textcite{ComaneciJoswig:24} computes the tropical median by solving an optimal transport problem.
  Depending on the field, the method for solving this problem differs.
  
  The code for tropical $k$-means clustering, specifically the clustering conducted in \zcref{ssec:apicomplexa} has been made available as a git repository accessible via
  \url{https://github.com/dmg-lab/clustering-phylogenetic-trees}.
  The computations reported were obtained using commit 97788af9.
  
  \subsection{Tropical \K-means clustering in floating point arithmetic}
  \Oscar\ and \Polymake\ are capable of computing tropical medians
  (and hence, through our implementation, tropical $k$-means clustering)
  in exact arithmetic; e.g., using rational numbers.
  However, the computational effort for computing the tropical $k$-means clustering quickly becomes prohibitive
  with growing dimension and number of points.
  Therefore, computing the tropical $k$-means clustering of a set of phylogenetic trees over rational numbers
  usually is not feasible.
  
  Instead, our computational experiments reported on in this section use floating point arithmetic (floats).
  The tropical medians algorithm in \Polymake\ calls \MCF\ \cite{mcf} for solving the underlying optimal transport problem
  over floats.
  
  While computationally efficient, the discrete nature of floats
  inevitably causes computational artifacts that do not occur in exact arithmetic.
  For example, the vectors representing trees will not lie exactly on the Bergman fan.
  
  To make this precise, we call a point $t \in \torus{n}$ an \emph{$\epsilon$-ultrametric} for $\epsilon > 0$ sufficiently small if 
  $(1-\epsilon)\max\{t_{ab}, t_{ac}, t_{bc}\} \leq \Mid(t_{ab}, t_{ac}, t_{bc})$ for all $a, b, c \in X$ where 
  \[\Mid(t_{ab}, t_{ac}, t_{bc}) \coloneqq t_{ab} + t_{ac} + t_{bc} - \max\{t_{ab}, t_{ac}, t_{bc}\} - \min\{t_{ab}, t_{ac}, t_{bc}\}\ .\]
  The set of $\epsilon$-ultrametrics is called the \emph{$\epsilon${-}thickening of the Bergman fan} $\bergman{N}_{\epsilon}$.
  Clearly, this set is full-dimensional which prohibits it from being a tropical linear space.
  In particular, the dimension of the Fermat--Weber set of a set of $\epsilon$-ultrametrics is not bounded by \zcref{eq:dim_trop_median} but instead by $\gcd(m, n) -1$ where $m$ is the number of $\epsilon$-ultrametrics; see \cite[Theorem 7]{ComaneciJoswig:24}.
  Moreover, we do not know if the $\epsilon$-thickening of the Bergman fan is tropically convex.

  \begin{question}\label{q:epsilon_ultra_metric}
    Does there exist $\epsilon > 0$ such that the $\epsilon$-thickening of the Bergman fan is tropically convex?
    If not, does it contain the tropical convex hull of $t_1, \ldots, t_m \in \bergman{N}_{\epsilon}$ if its dimension is smaller or equal than $N{-}2$?
  \end{question}
  
  An answer to this question would imply when the tropical median of a set of $\epsilon${-}ultrametrics is again an $\epsilon${-}ultrametric.
  Unfortunately, a full investigation is beyond the scope of this article.
  
  Due to the numerics of floating point numbers, it might happen that given a set of $\epsilon$-ultrametrics,
  \Polymake\ computes a $\delta$-ultrametric, $\delta \gg \epsilon$, as the tropical median of $\mathcal{S}$.
  We stress the fact that this does not disprove \zcref{q:epsilon_ultra_metric} as computations with floating point numbers are by definition not exact. 
  
  In order to keep these numerical instabilities from piling up, we employ the following correction scheme.
  Let $\epsilon_0 \approx 10^{-16}$ be difference between the floating point number one and the next largest floating point number.
  After the computation of a tropical median consensus tree $t = \t + \RR\one$,
  if $\t$ is a $\delta$-ultrametric only for $\delta > \sqrt{\epsilon_0}$,
  we replace $t$ by $t' = \t' + \RR\one$ where $\t' \in \RR^n$ is the largest ultrametric smaller than $\t$.
  Again, due to the discrete nature of floating point numbers, it might happen that $\t'$ is not exactly representable;
  however, it seems always possible to find $\t'$ as an $\epsilon$-ultrametric for $\epsilon < \sqrt{\epsilon_0}$.
  
    
  Last, we encountered empty Fermat--Weber sets in our experiment as they became overdetermined.
  While the \texttt{tropical\_median} function in \Oscar\ returns a point which is not ultrametric 
  we consider this point a mere artifact of floating point arithmetic and not of any further use.
  We only observed this when the number of trees $m$ was divisible by $n$ where the dimension of the Fermat-Weber set can be full-dimensional due to the use of floats.
  Note that by reducing the number of trees by one to $m-1$, the tropical median becomes unique for both rational and floating point numbers. 
  As the tropical median has shown to be very robust to single outliers \cite[Example 23]{ComaneciJoswig:24}, 
  our strategy is to compute the tropical median of a subset of $m-1$ points.
  
  We will now proceed to describe how our method performs when clustering a real dataset of 268 phylogenetic trees on 8 taxa.

  \subsection{Tropical \K-means clustering of apicomplexa gene trees}\label{ssec:apicomplexa}
  The experiment described below is conducted on a dataset of 268 phyolgenetic trees of 7 different species of apicomplexa and one outgroup, namely \emph{Tetrahymena thermophila} (Tt).
  The dataset has been reconstruced by Kuo, Wares and Kissinger \cite{KWK08}.
  The seven species of apicomplexa are
    \emph{Babesia bovis} (Bb),
    \emph{Cryptosporidium parvum} (Cp),
    \emph{Eimeria tenella} (Et),
    \emph{Plasmodium falciparum} (Pf),
    \emph{Plasmodium vivax} (Pv),
    \emph{Theileria annulata} (Ta) and
    \emph{Toxoplasma gondii} (Tg).
    
 Their tropical consensus tree was computed in \cite[Section 5.3]{ComaneciJoswig:24}.
 All 268 trees in the dataset are non-binary and of 17 different coarse types.
 We ran the tropical $k$-means clustering algorithm~\ref{al:kmedian} (with $k = 17$) $100$ times.
 The distribution of losses of the computed clusterings is displayed \zcref{fig:losses}.
 Of these, the clustering $\mathcal{C}$ with the smallest loss is displayed in \zcref{fig:k-means-apicomplexa}.
 The total number of trees per cluster in $\mathcal{C}$ is displayed in \zcref{tab:size-apicomplexa-cluster}.
 
 \begin{figure}
   \includegraphics[scale=.3]{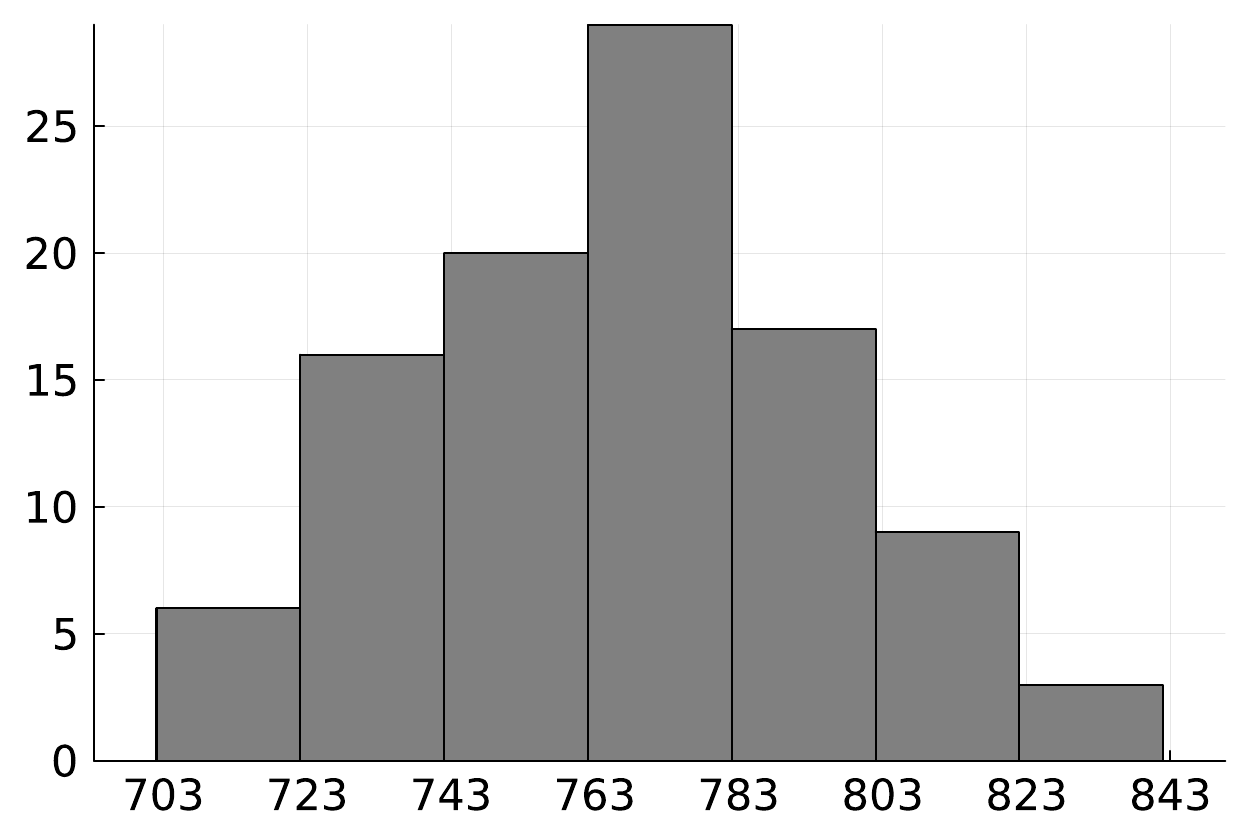}
   \caption{Losses of $100$ different tropical $17$-means++ clusterings of the apicomplexa dataset.}
   \label{fig:losses}
 \end{figure}

 \begin{figure}
 	\foreach \i in {1,...,17} {
 		\subcaptionbox*{\i}{
 				\includegraphics[width=.2\linewidth]{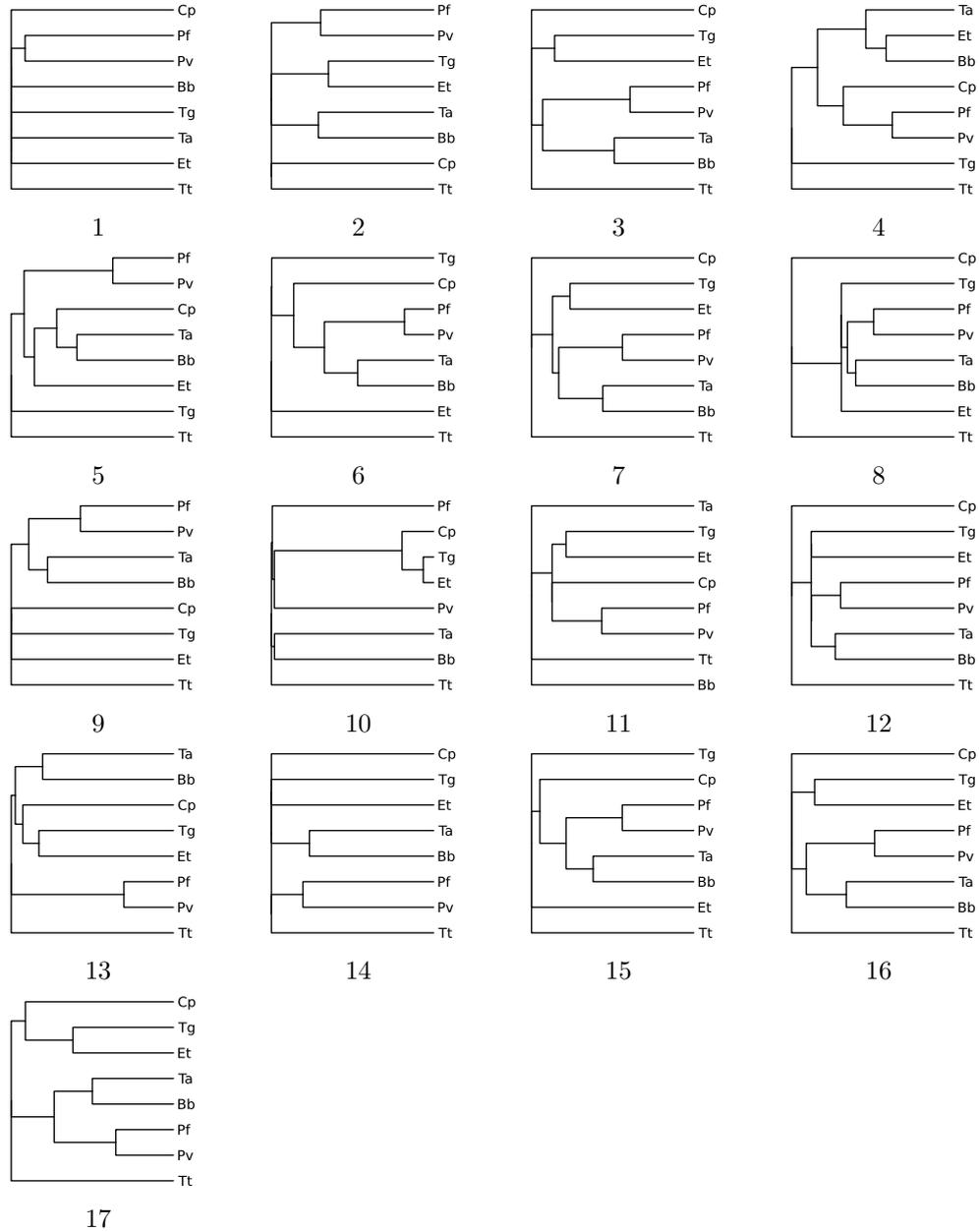}
 		}\hfill%
 	}
 	\caption{The set of tropical median consensus trees which are the final centroids of tropical $k$-means clustering for $k = 17$ with the smallest loss in \zcref{fig:losses}.}
 	\label{fig:k-means-apicomplexa}
 \end{figure}
 
 The resulting consensus trees are of 13 different coarse types.
 In particular, centroids 7, 8 and 12 have the same coarse type as well as 
 centroids 3 and 16,
 centroids 4 and 5 and 
 centroids 6 and 15 respectively. 
 
 \begin{table}
   \begin{tabular}{c|*{16}{c}c}
   	\toprule
     $i$ & 1 & 2 & 3 & 4 & 5 & 6 & 7 & 8 & 9 & 10 & 11 & 12 & 13 & 14 & 15 & 16 & 17 \\
     \midrule
     $|T_i|$ & 31 & 19 & 6 & 1 & 2 & 1 & 18 & 14 & 29 & 1 & 27 & 52 & 1 & 44 & 12 & 9 & 1 \\\bottomrule
   \end{tabular}
   \caption{The total number of trees in each cluster $T_i$ whose centroid is the consensus tree $i$ in \zcref{fig:k-means-apicomplexa}.}
   \label{tab:size-apicomplexa-cluster}
 \end{table}
 \smallskip
 
 Note that there are 5 different clusters, namely $T_4$, $T_6$, $T_{10}$, $T_{13}$ and $T_{17}$ which are singletons. 
 We call their elements \emph{lonely centroids}.
 We investigate whether they arise this way naturally as isolated outliers or whether there are too many clusters.
 Thus we reduce $k$ to 10 and again, compute 100 different clusterings.
 Among them, 69 have centroid 10 as a lonely centroid.
 Centroids 4, 6, 13 and 17 appear as lonely centroids in 13, 26, 6 and 3 of the 100 clusterings respectively. 
 We can therefore assume that they are all indeed isolated to different degrees.
 
 Before we begin to analyse the consensus trees in \zcref{fig:k-means-apicomplexa} we note that the tropical consensus method is known to be conservative, i.e., it avoids false positives; see \cite[Section 6]{ComaneciJoswig:24}.
 Additionally it is \emph{Pareto} and \emph{co-Pareto}. 
 That is, if three leaves corresponding to the taxa $a, b, c \in X$ form a \emph{rooted triplet} $ab|c$, i.e., $t_{ab} < t_{ac} = t_{bc}$ in every tree of the underlying dataset 
 then it is contained in the tropical median consensus tree.
 On the other hand, if a rooted triples appears in the consensus tree, then there exists at least one tree which contains the same rooted triplet. 

 The resulting consensus trees of tropical $k$-means++ clustering highlight the heterogeneity of the underlying dataset.
 Yet, their reduced number makes it easier to detect similarities as well as discrepancies.
 Consider for example the clade \textit{(Pf, Pv)}.
 It appears in all but one consensus tree which represent $99.6\%$ of the trees. 
 It implies a high support in the underlying dataset.
 However the timing of their split, measured as the depth of the node corresponding to their most recent common ancestor is almost uniformly distributed between $0.4$ and $1.14$ with one outlier at $0$ and another at $2.5$, see \zcref{fig:split-ambiguity-clustering}.
 
 The clade \textit{(Bb, Ta)} is found in all consensus trees but centroids 1, 4 and 11 which represent to $89.6\%$ of all trees.
 Again the timing of their split is varied widely but almost uniformly between with 2 outliers at $0$ and one close to $1.5$; see \zcref{fig:split-ambiguity-clustering}.
 
 \begin{figure}
   \includegraphics[scale=.2]{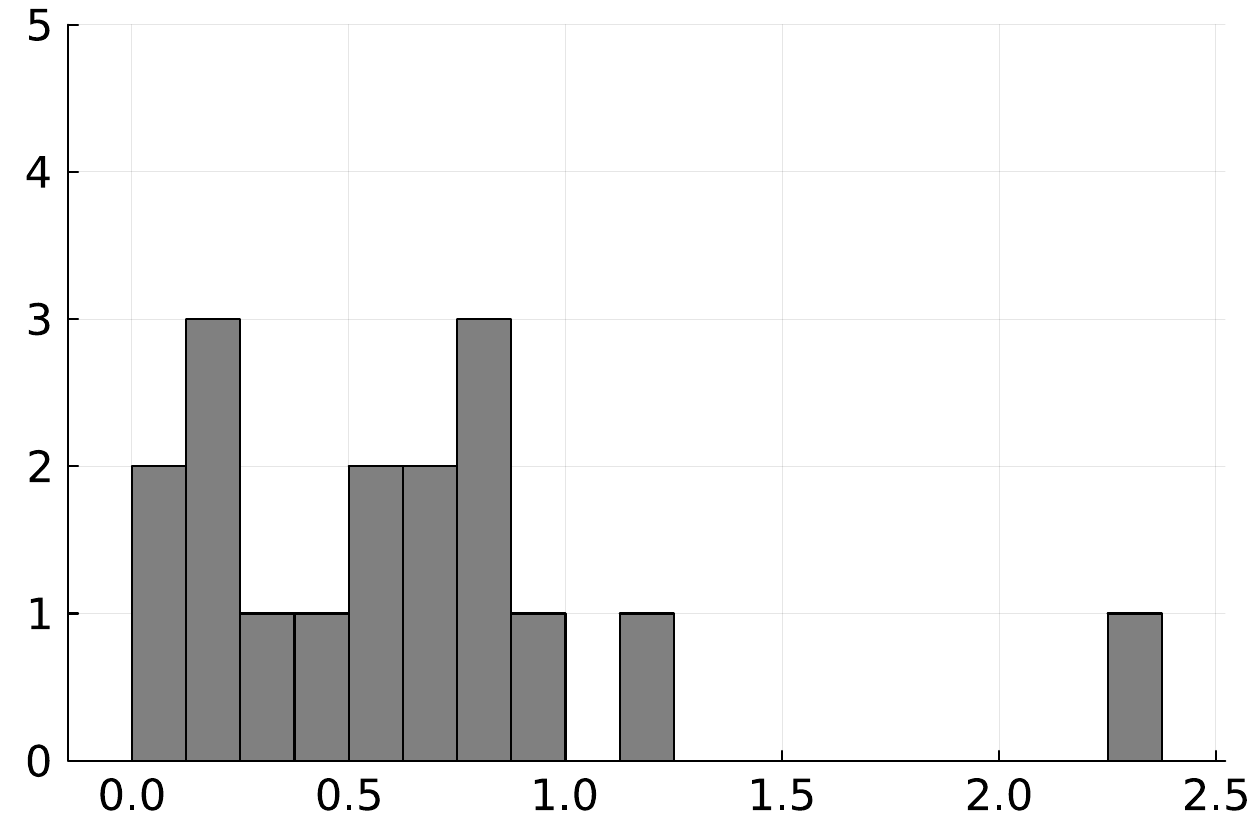} \qquad 
   \includegraphics[scale=.2]{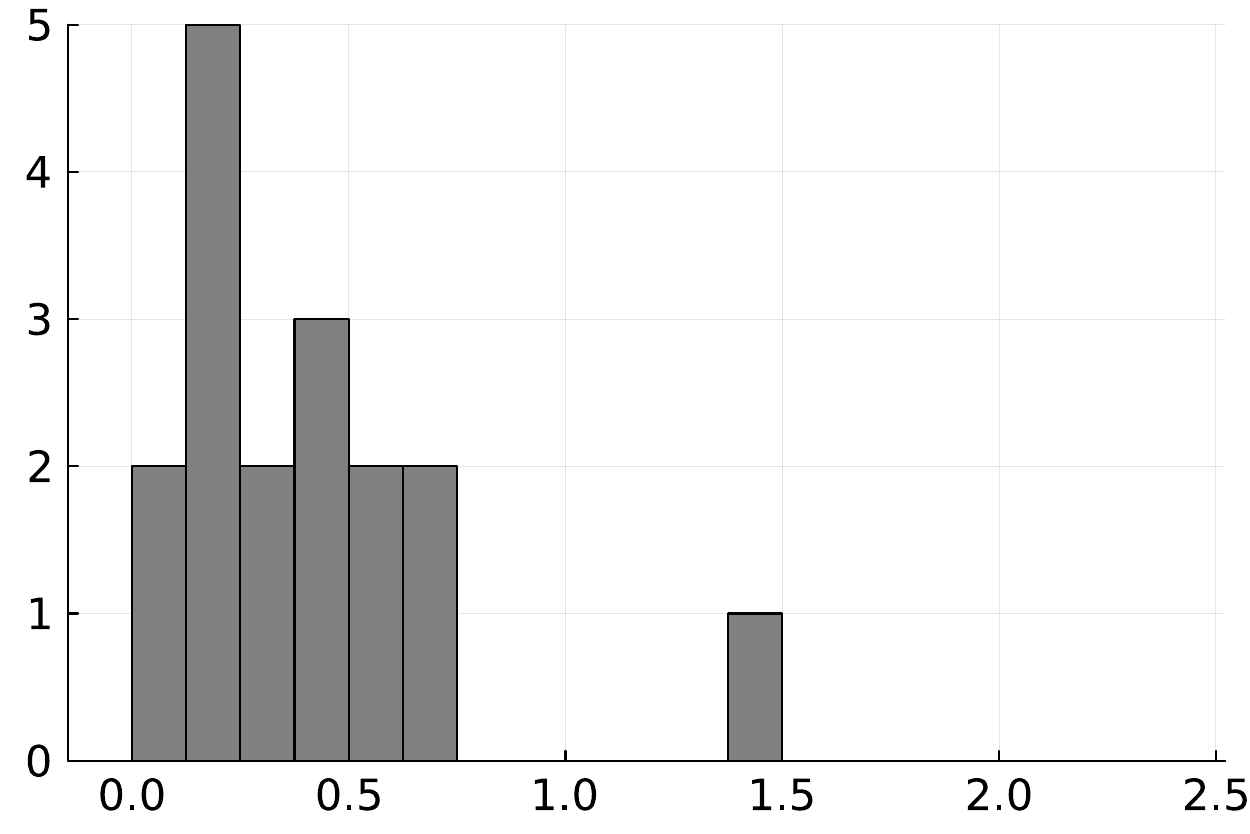}
   \caption{The distribution of the depth of the node corresponding to the most recent common ancestor of \textit{Pf} and \textit{Pv} (left) and \textit{Bb} and \textit{Ta} (right) in the centroids of $\mathcal{C}$ displayed in \zcref{fig:k-means-apicomplexa}.}
   \label{fig:split-ambiguity-clustering}
 \end{figure}
 
 For both clades, their strong relationship in the clustering is representative to the original dataset
 as $99.2\%$ of the sample trees contain the clade \textit{(Pf, Pv)} and $94\%$ contain the clade \textit{(Bb, Ta)}.
 
 Additionally we examine the support of the clade \emph{((Bb, Ta), (Pf, Pv))} as in the original paper by \textcite[Table 2]{KWK08}.
 In the clustering in \zcref{fig:k-means-apicomplexa} there are in total 90 trees belonging to a cluster whose consensus tree supports the aforementioned clade.
 This equates to 30\% compared to 44\% based on ML Consensus in \cite{KWK08}.
 Note that this is another indicator for the conservativeness of the tropical consensus method.
  
 Finally, we observe the clade \textit{(Et, Tg)} which appears in centroids 2, 3, 7, 10, 11, 13, 16, and 17 which represent $30.6\%$ of all trees.
 The relatively low percentage in representation is largely due to the fact that the clade does not appear in any of the four biggest clusters corresponding to centroids 8, 1, 14, 12.
 Taking all consensus trees as well as the size of the cluster they represent into consideration suggests that the split of \textit{(Et, Tg)} happened on average earlier than the split of the previous two groups or in fewer sample or both.
 In fact, the latter turns out to be true.
 The clade is supported by $87\%$ of trees in the original sample contain the clade \textit{(Et, Tg)}.
 \zcref{fig:split-ambiguity-sample} depicts the depth of the node corresponding to the most recent common ancestor of the three pairs.
   \begin{figure}
  	\includegraphics[scale=.2]{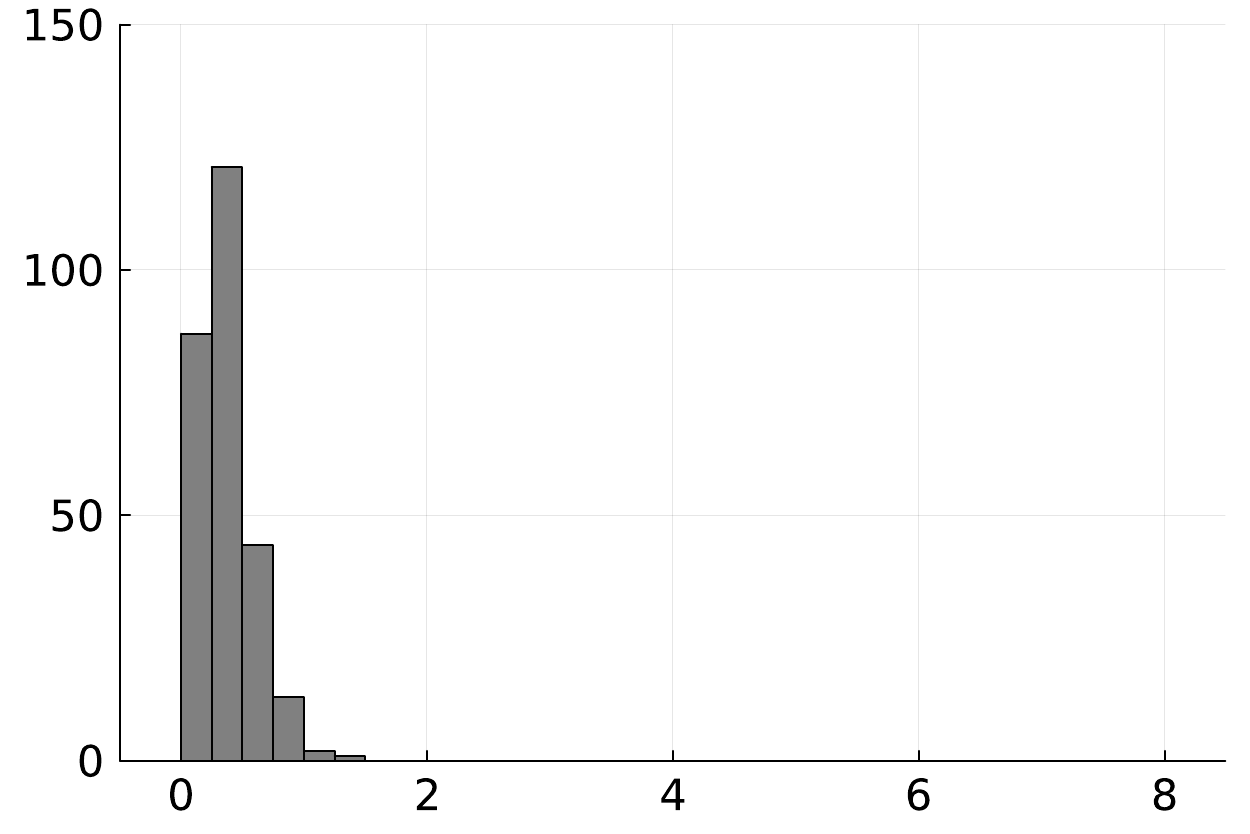}   
  	\includegraphics[scale=.2]{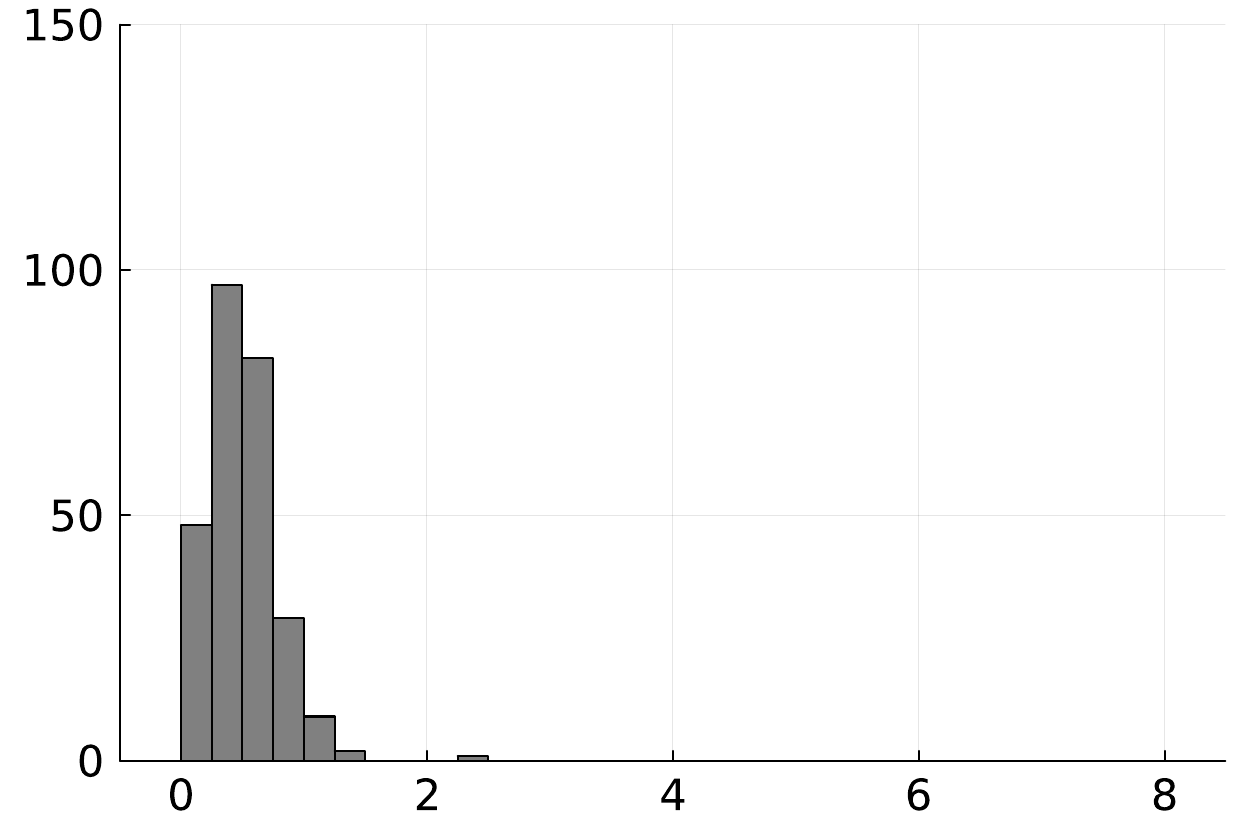}
  	\includegraphics[scale=.2]{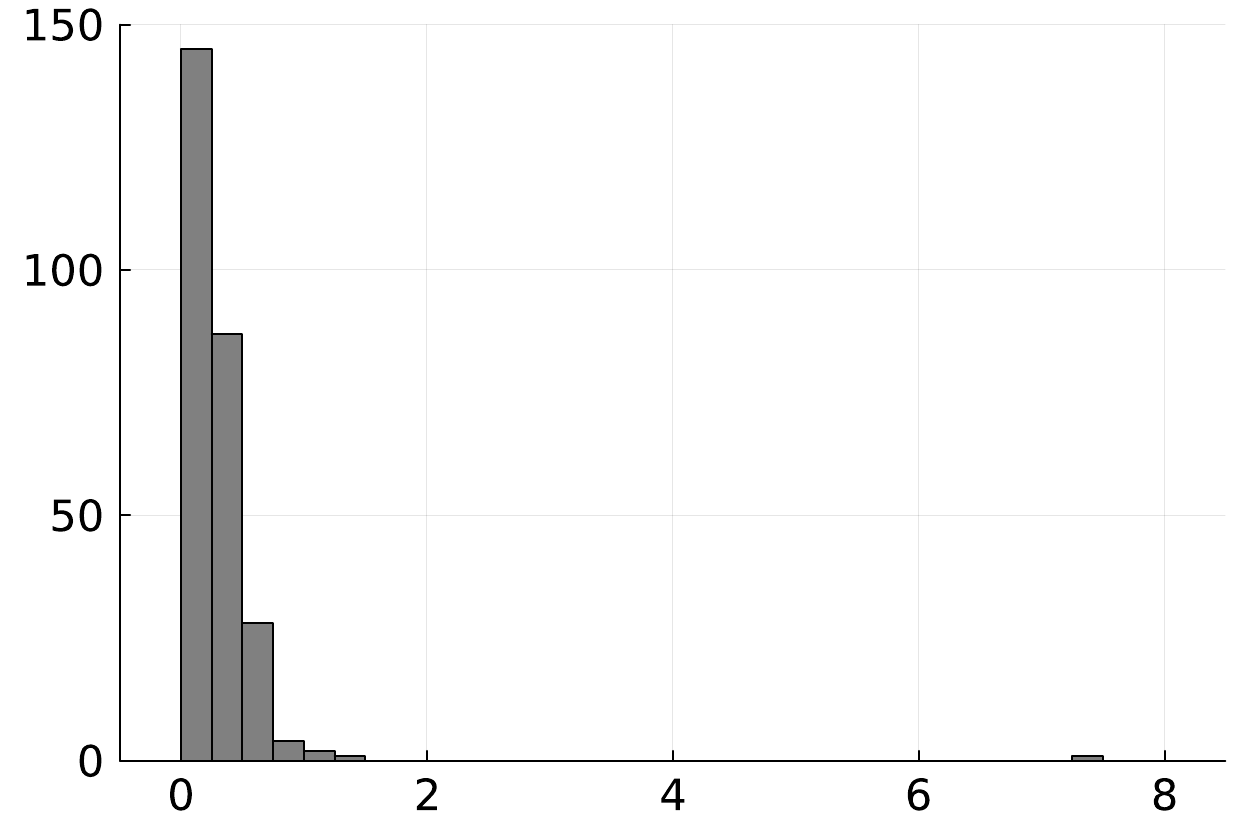}
  	\caption{The depth of the node corresponding to the most recent common ancestor of \textit{Bb} and \textit{Ta} (left), \textit{Pf} and \textit{Pv} (middle) as well as \textit{Et} and \textit{Tg} (right) in the original dataset of 268 trees.}
  	\label{fig:split-ambiguity-sample}
  \end{figure} 
    
  \section{Outlook}\label{sec:outlook}
  For a homogeneous dataset of phylogenetic trees with only a few different combinatorial types, the question whether we can recover a partition into combinatorial types becomes interesting.
  By a similar argument as in \zcref{thm:coarse-type-median} one can show that the maximal cones of the coarse subdivision of the Bergman fan $\bergman{N}$ are tropically convex.
  Assume that the trees $t$ are ``sufficiently binary'', i.e., there exists $\zeta > 0$ such that for every three taxa $a,b,c \in X$ we have $\max\{t_{ab},t_{ac}, t_{bc}\} - \min\{t_{ab},t_{ac},t_{bc}\} > \zeta$.
  Is there a lower bound for $\zeta$ and an upper bound on $\eta(t)$ for the trees $t$ such that we can recover the partition (into combinatorial types) with tropical $k$-means++ clustering? 
  
  In order to better analyse the resulting clustering produced by tropical $k$-means clustering, it is essential to understand all possible trees which can be contained in a given cluster.
  Since all such trees are contained in a ball around the corresponding centroid with respect to the tropical asymmetric metric,
  this motivates the study of the intersection of such a ball with the Bergman fan $\bergman{N}$.
  
%

  \small\sloppy
  \printbibliography
\end{document}